\documentclass[a4paper,11pt,twoside,notitlepage]{amsart}

\usepackage{amsmath,amssymb,amsthm,amscd}
\usepackage{mathtools}
\usepackage{mathrsfs}

\usepackage{graphicx}
\usepackage{xcolor}
\usepackage{tikz}

\usepackage[ocgcolorlinks,linkcolor=blue]{hyperref}
\usepackage{url}
\usepackage[shortlabels]{enumitem}
\usepackage{comment}
\usepackage{verbatim}
\usepackage{lineno}

\usepackage{bm}
\usepackage{bbm}
\usepackage{dsfont}
\usepackage{relsize}
\usepackage{esint}
\usepackage[normalem]{ulem}
\usepackage[utf8]{inputenc}

\newcommand\redsout{\bgroup\markoverwith{\textcolor{red}{\rule[0.5ex]{2pt}{0.4pt}}}\ULon}

\numberwithin{equation}{section}

\DeclareMathOperator{\dist}{dist}

\allowdisplaybreaks
\newcommand{\para}[1]{\vspace{3mm}\noindent\textbf{#1.}}

\mathtoolsset{showonlyrefs}
\graphicspath{{images/}}

\newtheorem{theorem}{Theorem}[section]
\newtheorem{lemma}[theorem]{Lemma}
\newtheorem*{lemma*}{Lemma}
\newtheorem{definition}[theorem]{Definition}
\newtheorem{corollary}[theorem]{Corollary}

\newtheorem{proposition}[theorem]{Proposition}
\newtheorem{question}{Question}
\newtheorem{remark}[theorem]{Remark}

\newcommand{\R}{\mathbb R}
\newcommand{\C}{\mathbb C}
\newcommand{\N}{\mathbb N}

\newcommand{\Sn}{\mathbb S^{n-1}}

\newcommand{\supp}{\operatorname{supp}}
\newcommand{\PV}{\operatorname{P.V.}}
\newcommand{\wt}{\widetilde}

\newcommand{\ip}[2]{\left\langle #1,#2\right\rangle}

\newcommand{\abs}[1]{\left\lvert #1\right\rvert}

\title[Recovering stable kernels]{Recovering stable kernels from exterior measurements}

\author[Y.-H. Lin]{Yi-Hsuan Lin}
\address{Department of Applied Mathematics, National Yang Ming Chiao Tung University, Hsinchu, Taiwan}
\email{yihsuanlin3@gmail.com}


\begin{document}

	\begin{abstract}
		We study an inverse problem for translation-invariant symmetric stable operators of the form
		\[
		L_a u(x)=\PV\int_{\R^n}(u(x)-u(y))\frac{a((x-y)/|x-y|)}{|x-y|^{n+2s}}\,dy,
		\quad 0<s<1,
		\]
		where the unknown is the even angular density $a$ on $\Sn$. For a bounded open set $\Omega\subset\R^n$, with $\Omega_e=\R^n\setminus\overline\Omega$, we consider restricted exterior Dirichlet-to-Neumann maps $\Lambda_a^{W_1,W_2}$, where exterior data are supported in $W_1\Subset\Omega_e$ and the nonlocal Neumann data are observed on $W_2\Subset\Omega_e$. We prove three recovery results for the leading angular density. In the overlapping regime $W_1\cap W_2\ne\emptyset$, the exterior diagonal singularity determines every smooth elliptic angular density. In the separated regime $\overline W_1\cap\overline W_2=\emptyset$, where this singularity is absent, we prove uniqueness in the finite harmonic angular class by an exact factorization of the stable symbol. We also prove separated-data uniqueness for real-analytic angular densities when the source and observation sets lie in the unbounded exterior component, using analytic continuation of the off-diagonal Dirichlet-to-Neumann kernel and a far-field asymptotic argument.
	\end{abstract}

	\maketitle
	\tableofcontents

	\section{Introduction}\label{sec:introduction}
	
	Nonlocal elliptic operators arise naturally as infinitesimal generators of jump processes and as models in anomalous diffusion. A distinguished example is the fractional Laplacian $(-\Delta)^s$, $0<s<1$, whose strong unique continuation property has played a central role in the fractional Calder\'on problem of \cite{GSU20}. Subsequent developments include uniqueness, reconstruction, stability, single measurement results, drift perturbations, nonlinear models, parabolic models and higher order fractional analogues; see, for instance, \cite{GRSU20,CLR20,ruland2015unique,Yu17,CMR20,CMRU20,LLR2019calder,KLW2021calder,lin2020monotonicity}. Variable coefficient and anisotropic versions of fractional inverse problems have also been studied in \cite{GLX,CGRU2023reduction,ruland2023revisiting,feizmohammadi2024fractional_closed,FGKU_2025,FKU24}. We also refer the reader to the recent monograph \cite{LL25_Integro} for further background.
	
	The purpose of this paper is to study the recovery of the leading angular density for a class of translation-invariant symmetric stable operators. Let $a$ be a real-valued even function on the unit sphere $\Sn$, and define
	\begin{equation}\label{eq:kernel-intro}
		K_a(z)=\frac{a(z/|z|)}{|z|^{n+2s}},\quad z\in\R^n\setminus\{0\}.
	\end{equation}
	The associated symmetric integro-differential operator is
	\begin{equation}\label{eq:operator-intro}
		L_a u(x):=\PV\int_{\R^n}(u(x)-u(y))K_a(x-y)\,dy,
	\end{equation}
	where $\PV$ denotes the Cauchy principal value. Since $L_a$ is translation invariant, it is a Fourier multiplier: for $u\in\mathcal S(\R^n)$, one has $\widehat{L_a u}(\xi)=A_a(\xi)\widehat u(\xi)$. The Fourier multiplier symbol is given by
	\begin{equation}\label{eq:symbol-intro}
		A_a(\xi)=c_{n,s}\int_{\Sn}|\xi\cdot\theta|^{2s}a(\theta)\,d\theta,
	\end{equation}
	where $c_{n,s}>0$ is fixed throughout the paper. We assume the ellipticity condition
	\begin{equation}\label{eq:ellipticity-intro}
		\lambda |\xi|^{2s}\le A_a(\xi)\le \Lambda |\xi|^{2s}
	\end{equation}
	for all $\xi\in\R^n$ and for some constants $0<\lambda\le\Lambda<\infty$. The isotropic density gives the fractional Laplacian up to a multiplicative constant. In general, $a$ describes the anisotropic angular profile of the jumps.
	
	The homogeneous stable operators in \eqref{eq:operator-intro} are closely related to the \emph{directionally antilocal operators} studied in \cite{CoviGarciaFerreroRuland2022}. In that work, the principal homogeneous operator is fixed, and lower-order perturbations are recovered by combining directional antilocality, domain of dependence geometry, and Runge approximation. In contrast, the present paper treats the principal angular density itself as the unknown. Thus, the inverse problem considered here is a leading coefficient problem rather than a lower-order perturbation problem.
	
	This distinction is important. In the recovery of a lower-order potential, the usual Alessandrini identity contains an interior product of solutions. In the present problem, the unknown appears in the leading jump kernel and is attached to the direction $(x-y)/|x-y|$. Thus, the natural integral identity involves angular jump products rather than pointwise products in $\Omega$. This is the main reason why the leading coefficient problem is not a direct consequence of the Runge approximation methods used for lower-order terms.
	
	Let $\Omega\subset\R^n$ be bounded and open, and set $\Omega_e:=\R^n\setminus\overline\Omega$. Given an exterior datum supported in an open set $W_1\Subset\Omega_e$, let $u_f$ solve
	\begin{equation}\label{eq:intro-dirichlet-cases}
		\begin{cases}
			L_a u_f=0 & \text{in }\Omega,\\
			u_f=f & \text{in }\Omega_e.
		\end{cases}
	\end{equation}
	The exterior condition in \eqref{eq:intro-dirichlet-cases} is understood through the Sobolev condition $u_f-F\in\widetilde H^s(\Omega)$, where $F\in H^s(\R^n)$ is a representative of the exterior datum. For a second exterior open set $W_2\Subset\Omega_e$, we consider the restricted exterior Dirichlet-to-Neumann (DN) map
	\begin{equation}\label{eq:DN-intro-space}
		\Lambda_a^{W_1,W_2}:\widetilde H^s(W_1)\to H^{-s}(W_2).
	\end{equation}
	The precise definition of the Sobolev spaces and of the map \eqref{eq:DN-intro-space} is given in Section \ref{sec:preliminaries}. We are interested in the following question:
	
	\begin{question}\label{Q:IP}	
		Does $\Lambda_a^{W_1,W_2}$ determine $a$?
	\end{question}

	There are two different measurement regimes, illustrated in Figure \ref{fig:measurement-geometries}. If $W_1\cap W_2\ne\emptyset$, then the exterior diagonal singularity of the restricted DN map is visible. In that case, the leading kernel can be recovered directly. The more delicate case is the genuinely separated regime
	\begin{equation}\label{eq:separated-intro}
		\overline W_1\cap\overline W_2=\emptyset.
	\end{equation}
	In this case, the direct kernel $K_a(x-y)$ is smooth on $W_2\times W_1$, and the diagonal singularity is not measured. The separated problem asks whether the remaining off-diagonal information still determines the angular density. This question is the main motivation of the paper.
	
	\begin{figure}[t]
		\centering
		\begin{tikzpicture}[scale=0.90, every node/.style={font=\small}]
			
			\begin{scope}[shift={(-4.55,0)}]
				
				\draw[fill=gray!4, draw=gray!35, rounded corners=3pt]
				(-3.55,-2.10) rectangle (3.55,2.30);
				\node at (0,2.02) {\textbf{Overlapping exterior measurements}};
				\node[gray!70!black] at (-2.55,1.55) {$\Omega_e$};
				
				\draw[fill=orange!35, draw=orange!80!black, thick]
				(-1.35,0) ellipse (0.82 and 0.60);
				\node at (-1.35,0) {$\Omega$};
				
				\draw[fill=blue!22, draw=blue!70!black, thick]
				(1.35,0.28) ellipse (0.92 and 0.56);
				\draw[fill=green!24, draw=green!55!black, thick]
				(1.98,0.02) ellipse (0.92 and 0.56);
				
				\node[blue!70!black] at (0.88,-0.53) {$W_1$};
				\node[green!45!black] at (2.42,-0.66) {$W_2$};
				
				\draw[fill=purple!25, draw=purple!70!black, thick]
				(1.68,0.14) ellipse (0.28 and 0.20);
				\node[purple!70!black] at (2.10,0.00) {$U$};
				\fill[purple!80!black] (1.68,0.14) circle (1.5pt);
				\node[purple!80!black] at (1.8,0.50) {\scriptsize exterior diagonal $x=y$};
				\draw[->, purple!70!black] (2.05,0.40) -- (1.72,0.18);
				
				\node at (1.28,-1.35) {$U\Subset W_1\cap W_2$};
				\node at (1.28,-1.78) {$W_1\cap W_2\ne\emptyset$};
				
				\draw[<->, gray!70] (-0.52,0.42) -- (0.42,0.55);
				\node[gray!70!black] at (-0.05,1.00) {\scriptsize exterior sets stay away from $\Omega$};
				
			\end{scope}

			\begin{scope}[shift={(4.15,0)}]
				
				\draw[fill=gray!4, draw=gray!35, rounded corners=3pt]
				(-3.55,-2.10) rectangle (3.35,2.30);
				\node at (0,2.02) {\textbf{Separated exterior measurements}};
				\node[gray!70!black] at (-2.85,1.55) {$\Omega_e$};
				
				\draw[fill=orange!35, draw=orange!80!black, thick]
				(0,0) ellipse (0.82 and 0.60);
				\node at (0,0) {$\Omega$};
				
				\draw[fill=blue!22, draw=blue!70!black, thick]
				(-2.25,0.68) ellipse (0.72 and 0.46);
				\node[blue!70!black] at (-2.25,1.30) {$W_1$};
				\fill[blue!80!black] (-2.25,0.68) circle (1.5pt);
				\node at (-2.25,0.38) {\scriptsize $y$};
				
				\draw[fill=green!24, draw=green!55!black, thick]
				(2.25,0.68) ellipse (0.72 and 0.46);
				\node[green!45!black] at (2.25,1.30) {$W_2$};
				\fill[green!50!black] (2.25,0.68) circle (1.5pt);
				\node at (2.25,0.38) {\scriptsize $x$};
				
				\draw[->, thick, blue!70!black]
				(-1.62,0.90) .. controls (-0.55,1.55) and (0.55,1.55) .. (1.62,0.90);
				\node[blue!70!black] at (0,1.55) {\scriptsize direct term $K_a(x-y)$};
				
				\draw[->, thick, dashed, gray!75]
				(-1.80,0.50) .. controls (-1.00,-0.10) and (-0.58,-0.28) .. (-0.30,-0.20);
				\draw[->, thick, dashed, gray!75]
				(0.30,-0.20) .. controls (0.58,-0.28) and (1.00,-0.10) .. (1.80,0.50);
				\node[gray!75!black] at (0,-0.82) {\scriptsize interior correction through $\Omega$};
				
				\node at (0,-1.40) {$\overline W_1\cap\overline W_2=\emptyset$};
				\node at (0,-1.78) {\scriptsize the exterior diagonal $x=y$ is not measured};
				
			\end{scope}
			
		\end{tikzpicture}
		\caption{Exterior measurement geometries. The source set $W_1$ and the observation set $W_2$ are compactly contained in $\Omega_e=\R^n\setminus\overline\Omega$ and therefore do not touch $\Omega$. In the overlapping case, one may restrict to $U\Subset W_1\cap W_2$, where the exterior diagonal singularity of the Dirichlet-to-Neumann map is visible. In the separated case, the diagonal is not measured; the off-diagonal kernel contains a direct exterior-to-exterior contribution and an interior correction through $\Omega$.}
		\label{fig:measurement-geometries}
	\end{figure}
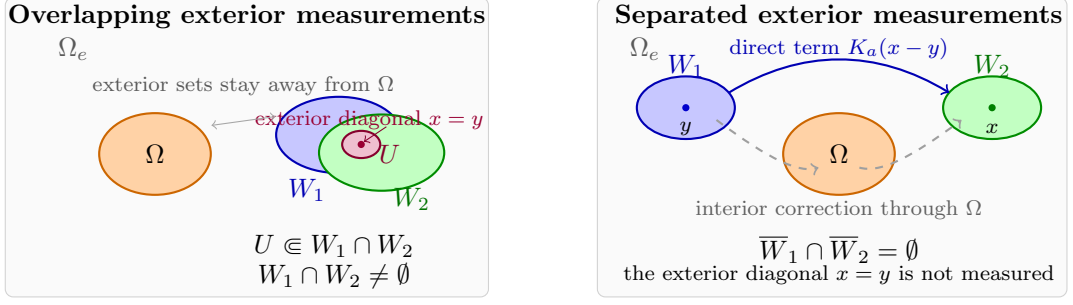
	
	For cone-supported directionally antilocal kernels, the domain of dependence of the operator leads to natural geometric restrictions and possible gauges; this is one of the main points of \cite{CoviGarciaFerreroRuland2022}. The classes considered in the main separated results below have full angular support. Thus, the geometric obstruction is not a cone-domain gauge. The difficulty is instead that the diagonal singularity is absent from the measured operator when \eqref{eq:separated-intro} holds.
	
	\para{Overlapping measurements}
	We first state the recovery result in the overlapping regime. We say that $a$ is a smooth elliptic angular density if $a\in C^\infty(\Sn)$ is real-valued, even, nonnegative, and the symbol $A_a$ satisfies \eqref{eq:ellipticity-intro}.
	
	The first result is the following.
	
	\begin{theorem}[Recovery from overlapping exterior measurements]\label{thm:overlap}
		Let $0<s<1$ and $n\ge2$. Let $\Omega\subset\R^n$ be bounded and open. Let $W_1,W_2\Subset\Omega_e$ be nonempty open sets such that $W_1\cap W_2\ne\emptyset$. Let $a_1,a_2\in C^\infty(\Sn)$ be smooth elliptic angular densities. If
		\begin{equation}\label{eq:overlap-DN-equality-intro}
			\Lambda_{a_1}^{W_1,W_2}f=\Lambda_{a_2}^{W_1,W_2}f \quad \text{for all }f\in \wt H^s(W_1),
		\end{equation}
		then $a_1=a_2$ on $\Sn$.
	\end{theorem}
	
	\para{Finite harmonic class and separated measurements}
	Let $\mathscr H_\ell$ denote the space of spherical harmonics of degree $\ell$ on $\Sn$. For $m\in\N_0$, set
	\begin{equation}\label{eq:Am-def}
		\mathcal A_m:=\bigoplus_{\substack{0\le \ell\le 2m\\ \ell\text{ even}}}\mathscr H_\ell.
	\end{equation}
	Thus, $\mathcal A_m$ is the finite-dimensional space of even spherical harmonics of degree at most $2m$. We do not include either the parity or the degree bound in the symbol $\mathcal A_m$. The restriction to even degrees $\ell$ is forced by the symmetry assumption. Indeed, spherical harmonics of degree $\ell$ satisfy $Y_\ell(-\theta)=(-1)^\ell Y_\ell(\theta)$. Hence, an even angular density has vanishing spherical harmonic coefficients in all odd degrees. Thus, the natural finite-dimensional class of even angular densities of degree at most $2m$ is precisely $\mathcal A_m$. This parity is also compatible with the factorization below: for even $\ell\le 2m$, the factor $|\xi|^{2m-\ell}$ is a polynomial in $\xi$.
	
	\begin{definition}\label{def:admissible}
		Let $0<s<1$ and $m\in\N_0$. A function $a$ is called an admissible finite harmonic density of order $m$ if
		\begin{enumerate}[\rm(i)]
			\item $a\in\mathcal A_m$ is real-valued;
			\item $a(\theta)\ge0$ for a.e. $\theta\in\Sn$;
			\item the symbol $A_a$ satisfies the ellipticity estimate \eqref{eq:ellipticity-intro}.
		\end{enumerate}
	\end{definition}
	
	The finite harmonic class is a full angular support finite-dimensional class. It is not intended to model kernels supported in a proper cone. Indeed, a real analytic angular function on the whole sphere cannot be supported in a proper open spherical cone unless it is identically zero.
	
	\begin{theorem}[Recovery from separated exterior measurements]\label{thm:separated-main}
		Let $0<s<1$ and $n\ge2$. Let $\Omega\subset\R^n$ be bounded and open. Let $W_1,W_2\Subset\Omega_e$ be nonempty open sets satisfying \eqref{eq:separated-intro}. Let $a_1,a_2\in\mathcal A_m$ be admissible finite harmonic densities of the same order $m$. If
		\begin{equation}\label{eq:operator-DN-equality-main}
			\Lambda_{a_1}^{W_1,W_2}f=\Lambda_{a_2}^{W_1,W_2}f
			\quad \text{for all }f\in\widetilde H^s(W_1),
		\end{equation}
		then $a_1=a_2$ on $\Sn$.
	\end{theorem}
	
	For the proof, it is enough to assume that the equality holds in $\mathcal D'(W_2)$ for all $f\in C_c^\infty(W_1)$. Although Theorem \ref{thm:separated-main} is stated for genuinely separated sets, this causes no loss in the finite harmonic class after restricting the measurements. Indeed, given nonempty exterior open sets $W_1,W_2\Subset\Omega_e$, one can choose nonempty open subsets $U_1\Subset W_1$ and $U_2\Subset W_2$ such that $\overline U_1\cap\overline U_2=\emptyset$. Equality of the DN maps on $W_1,W_2$ restricts to equality on $U_1,U_2$, and Theorem \ref{thm:separated-main} then applies. This is recorded later in Corollary \ref{cor:any-exterior-sets}.
	
	\para{Analytic angular densities}
	We also consider the separated problem for real analytic angular densities. We say that $a$ is a real analytic elliptic angular density if $a\in C^\omega(\Sn)$ is real-valued, even, nonnegative, and the symbol $A_a$ satisfies \eqref{eq:ellipticity-intro}. This is again a full angular support class.
	
	\begin{theorem}[Recovery of analytic angular densities]\label{thm:analytic-main}
		Let $0<s<1$ and $n\ge2$. Let $\Omega\subset\R^n$ be bounded and open, and let $G$ be the unbounded connected component of $\Omega_e$. Let $W_1,W_2\Subset G$ be nonempty open sets satisfying \eqref{eq:separated-intro}. Let $a_1,a_2\in C^\omega(\Sn)$ be real analytic elliptic angular densities. If
		\begin{equation}\label{eq:analytic-DN-equality-main}
			\Lambda_{a_1}^{W_1,W_2}f=\Lambda_{a_2}^{W_1,W_2}f
			\quad \text{for all }f\in\widetilde H^s(W_1),
		\end{equation}
		then $a_1=a_2$ on $\Sn$.
	\end{theorem}
	
	\begin{remark}\label{rem:arbitrary-windows-intro}
		Theorem \ref{thm:overlap}, Theorem \ref{thm:separated-main} and Theorem \ref{thm:analytic-main} use three different mechanisms. Overlapping measurements recover the leading kernel from the exterior diagonal singularity. In the finite harmonic separated case, the diagonal singularity is absent, but an algebraic factorization reduces the problem to unique continuation for $(-\Delta)^s$. In the analytic separated case, the off-diagonal DN kernel is real analytic in the exterior variables, and its vanishing propagates by analytic continuation to the far field.
	\end{remark}
	
	\para{Mechanism of the proof}
	The proof of Theorem \ref{thm:separated-main} uses the finite harmonic assumption. If $a\in\mathcal A_m$, the Funk-Hecke formula (see Section \ref{sec:factorization}) implies that there exists a homogeneous polynomial $Q_a$ of degree $2m$ such that
	\begin{equation}\label{eq:factor-intro}
		|\xi|^{2m}A_a(\xi)=|\xi|^{2s}Q_a(\xi).
	\end{equation}
	Equivalently,
	\begin{equation}\label{eq:operator-factor-intro}
		(-\Delta)^mL_a=(-\Delta)^sQ_a(D)
	\end{equation}
	as Fourier multipliers. Here, $Q_a(D)$ is the local constant coefficient differential operator with Fourier symbol $Q_a(\xi)$. This is the only point at which the fractional Laplacian appears. We do not claim that a general stable operator is a fractional Laplacian. Rather, finite harmonicity converts the multiplier after applying the local operator $(-\Delta)^m$.
	
	Let $u_j$ solve \eqref{eq:intro-dirichlet-cases} with $a=a_j$ and common exterior datum $f\in C_c^\infty(W_1)$. Since $f$ vanishes on $W_2$, one has $u_1=u_2=0$ in $W_2$. Equality of the DN maps gives $L_{a_1}u_1=L_{a_2}u_2$ in $W_2$. Applying the local operator $(-\Delta)^m$ and using \eqref{eq:operator-factor-intro}, we obtain 
	\[
	(-\Delta)^s(Q_{a_1}(D)u_1-Q_{a_2}(D)u_2)=0 \quad \text{in }W_2.
	\]
	The expression in parentheses also vanishes in $W_2$ because the operators $Q_{a_j}(D)$ are local. The strong unique continuation property for $(-\Delta)^s$ implies that this expression vanishes in all of $\R^n$. Restricting the identity to $W_1$, where $u_1=u_2=f$, gives $Q_{a_1}(D)f=Q_{a_2}(D)f$ in $W_1$ for all $f\in C_c^\infty(W_1)$. Therefore $Q_{a_1}=Q_{a_2}$, and the finite cosine transform injectivity gives $a_1=a_2$.
	
	The proof of Theorem \ref{thm:analytic-main} uses a different mechanism. For analytic angular densities, the off-diagonal kernel of the DN map is real analytic on the exterior pair domain away from the diagonal. Equality of the restricted DN maps gives equality of these kernels on $W_2\times W_1$. Analytic continuation then gives equality on the unbounded exterior pair component. A far-field limit separates the direct kernel $K_{a_1-a_2}(x-y)$ from the interior correction and recovers $a_1-a_2$.
	
	\para{Relation to other work}
	The finite harmonic proof is related in spirit to entanglement principles for fractional Laplacians \cite{FL24,FKU24,lai2025entanglement,lin2026entanglement}, but the mechanism is different. Entanglement principles separate different fractional powers. Here, all angular modes have the same homogeneity. The finite harmonic class is a favorable resonant case: after multiplying by a local power of $-\Delta$, the angular multiplier becomes local, and one can use the usual fractional unique continuation property. The analytic result is closer in spirit to classical analytic continuation arguments in inverse problems: once the off-diagonal DN kernel is known on a nonempty exterior product set, analyticity propagates this information to a far-field region where the leading kernel can be read off.
	
	There is also a related line of work on leading coefficient determination for nonlocal equations. In the elliptic setting, the Caffarelli--Silvestre extension has been used to relate the fractional Calder\'on problem with variable coefficients to the corresponding local Calder\'on problem \cite{CGRU2023reduction}. On closed Riemannian manifolds, spectral and heat semigroup methods have led to fractional anisotropic Calder\'on results and to the recovery of geometric data from local source-to-solution maps \cite{feizmohammadi2024fractional_closed,FGKU_2025,FKU26Schrodinger,Lin24parabolic}. Related reductions from nonlocal parabolic measurements to local parabolic Cauchy data were developed in \cite{LLU23newreduction}. These works recover local or geometric leading coefficients through extension, semigroup, spectral, or reduction mechanisms. The present paper is different in that the operator is translation invariant, but the leading stable kernel itself is unknown. Thus, the coefficient to be determined is not a local metric or conductivity in an extension problem, but the angular density in the jump kernel. The finite harmonic and analytic arguments below provide two mechanisms for extracting this angular density from exterior measurements.
	
	\para{Organization of the paper}
	Section \ref{sec:preliminaries} gives the functional setting, the well-posedness of the exterior value problem, and the restricted DN map. Section \ref{sec:factorization} proves the finite harmonic factorization and the injectivity results for the associated cosine transform. Section \ref{sec:analytic-kernel} proves the off-diagonal kernel representation and the analytic continuation tools used for Theorem \ref{thm:analytic-main}. Section \ref{sec:further} is devoted to the inverse problems: it records the leading coefficient Alessandrini identity, discusses the structural obstruction for general smooth angular densities, and proves the three recovery theorems in Subsection \ref{sec:proofs}.

	\section{Function spaces, stable operators and exterior measurements}\label{sec:preliminaries}
	
	In this section, we fix the functional framework for the exterior value problem and for the restricted DN maps. Since the finite harmonic class considered in this paper has full angular support, we work with the full exterior region $\Omega_e=\R^n\setminus\overline\Omega$.
	
	\subsection{Fractional Sobolev spaces}
	We use the unitary Fourier transform
	\[
	\widehat u(\xi)=(2\pi)^{-n/2}\int_{\R^n}e^{-\mathsf ix\cdot\xi}u(x)\,dx,
	\]
	where $\mathsf i=\sqrt{-1}$. For $t\in\R$, let $H^t(\R^n)$ be the standard $L^2$-based Sobolev space with norm 
	\[
	\|u\|_{H^t(\R^n)}=\|(1+|\xi|^2)^{t/2}\widehat u(\xi)\|_{L^2(\R^n)}.
	\]
	If $U\subset\R^n$ is open, define $H^t(U):=\{u|_U:u\in H^t(\R^n)\}$ and $\widetilde H^t(U):=\overline{C_c^\infty(U)}^{H^t(\R^n)}$. The space $H^t(U)$ is equipped with the quotient norm
	\[
	\|v\|_{H^t(U)}=\inf\{\|V\|_{H^t(\R^n)}:V\in H^t(\R^n),\ V|_U=v\}.
	\]
	For $t\ge0$, we set $H^{-t}(U):=(\widetilde H^t(U))^*$. We identify elements of $\widetilde H^t(U)$ with their zero extensions to $\R^n$. The duality pairing between $H^{-t}(U)$ and $\widetilde H^t(U)$ is denoted by $\ip{\cdot}{\cdot}_{H^{-t}(U),\widetilde H^t(U)}$. We use the following convention for complex duality pairings. If $F\in H^{-t}(U)$ and $v\in\widetilde H^t(U)$, then $\ip{F}{v}_{H^{-t}(U),\widetilde H^t(U)}$ is linear in $F$ and conjugate-linear in $v$. Equivalently, for $c\in\C$, $\ip{cF}{v}=c\ip{F}{v}$, $\ip{F}{cv}=\overline c\ip{F}{v}$. When the spaces are clear from the context, we simply write $\ip{F}{v}$.
	
	\subsection{The energy form and the operator}
	Let $a\in L^1(\Sn)$ be real-valued and even. For $u,v\in\mathcal S(\R^n)$ define
	\begin{equation}\label{eq:energy-def}
		\mathcal E_a(u,v)=
		\frac12\iint_{\R^n\times\R^n}
		(u(x)-u(y))\overline{(v(x)-v(y))}
		\frac{a((x-y)/|x-y|)}{|x-y|^{n+2s}}\,dx\,dy.
	\end{equation}
	Thus $\mathcal E_a$ is linear in the first argument and conjugate-linear in the second. The following lemma fixes the normalization of $A_a$.
	
	\begin{lemma}\label{lem:fourier-representation}
		Let $a\in L^1(\Sn)$ be real-valued and even. Then, for all $u,v\in\mathcal S(\R^n)$,
		\begin{equation}\label{eq:energy-fourier}
			\mathcal E_a(u,v)=\int_{\R^n}A_a(\xi)\widehat u(\xi)\overline{\widehat v(\xi)}\,d\xi.
		\end{equation}
		If \eqref{eq:ellipticity-intro} holds, then $\mathcal E_a$ extends uniquely to a bounded sesquilinear form on $H^s(\R^n)\times H^s(\R^n)$ and
		\begin{equation}\label{eq:coercive-Hs-seminorm}
			\lambda\int_{\R^n}|\xi|^{2s}|\widehat u(\xi)|^2\,d\xi
			\le
			\mathcal E_a(u,u)
			\le
			\Lambda\int_{\R^n}|\xi|^{2s}|\widehat u(\xi)|^2\,d\xi.
		\end{equation}
	\end{lemma}
	
	\begin{proof}
		Let $K_a$ be defined by \eqref{eq:kernel-intro}. Since $a$ is even, $K_a(z)=K_a(-z)$. With the change of variables $z=x-y$,
		\[
		\mathcal E_a(u,v)
		=
		\frac12\int_{\R^n}\int_{\R^n}
		(u(x)-u(x-z))\overline{(v(x)-v(x-z))}
		K_a(z)\,dx\,dz.
		\]
		By Plancherel's formula, there holds 
		\[
		\int_{\R^n}(u(x)-u(x-z))\overline{(v(x)-v(x-z))}\,dx=\int_{\R^n}|1-e^{-\mathsf iz\cdot\xi}|^2\widehat u(\xi)\overline{\widehat v(\xi)}\,d\xi.
		\]
		Since $|1-e^{-\mathsf it}|^2=2(1-\cos t)$, Fubini gives
		\[
		\mathcal E_a(u,v)=\int_{\R^n}\bigg(\int_{\R^n}(1-\cos(z\cdot\xi))K_a(z)\,dz\bigg)
		\widehat u(\xi)\overline{\widehat v(\xi)}\,d\xi.
		\]
		Passing to polar coordinates $z=r\theta$ in the inner integral gives
		\[
		\int_{\R^n}(1-\cos(z\cdot\xi))K_a(z)\,dz=\int_{\Sn}a(\theta)\int_0^\infty
		\frac{1-\cos(r\,\xi\cdot\theta)}{r^{1+2s}}\,dr\,d\theta.
		\]
		For each fixed $\theta$, the one-dimensional integral is equal to $|\xi\cdot\theta|^{2s}\int_0^\infty\frac{1-\cos \rho}{\rho^{1+2s}}\,d\rho$,
		with the obvious value $0$ when $\xi\cdot\theta=0$. Hence
		\[
		\int_{\R^n}(1-\cos(z\cdot\xi))K_a(z)\,dz=\bigg(\int_0^\infty\frac{1-\cos \rho}{\rho^{1+2s}}\,d\rho\bigg)\int_{\Sn}|\xi\cdot\theta|^{2s}a(\theta)\,d\theta.
		\]
		The one-dimensional constant is positive and finite for $0<s<1$, and it is absorbed into $c_{n,s}$. This proves \eqref{eq:energy-fourier}. The inequalities \eqref{eq:coercive-Hs-seminorm} follow from \eqref{eq:ellipticity-intro}. The upper bound and the density of $\mathcal S(\R^n)$ in $H^s(\R^n)$ give the extension of $\mathcal E_a$ to $H^s(\R^n)\times H^s(\R^n)$.
	\end{proof}
	
	For $u\in H^s(\R^n)$, the distribution $L_a u\in H^{-s}(\R^n)$ is defined by
	\begin{equation}\label{eq:operator-duality}
		\ip{L_a u}{\varphi}:=\mathcal E_a(u,\varphi),\quad \varphi\in H^s(\R^n).
	\end{equation}
	By Lemma \ref{lem:fourier-representation}, this agrees with the Fourier multiplier with symbol $A_a$.
	
	\begin{lemma}[Fractional Poincar\'e inequality]\label{lem:poincare}
		Let $\Omega\subset\R^n$ be bounded. There is a constant $C=C(n,s,\Omega)>0$ such that
		\begin{equation}\label{eq:poincare}
			\|w\|_{L^2(\R^n)}
			\le C\bigg(\int_{\R^n}|\xi|^{2s}|\widehat w(\xi)|^2\,d\xi\bigg)^{1/2},
			\quad w\in\widetilde H^s(\Omega).
		\end{equation}
	\end{lemma}
	
	\begin{proof}
		It is enough to prove the estimate for $w\in C_c^\infty(\Omega)$ and then pass to the closure. Choose $R>0$ such that $\Omega\subset B_R(0)$. Since $w=0$ in $\R^n\setminus\Omega$,
		\[
		\iint_{\R^n\times\R^n}
		\frac{|w(x)-w(y)|^2}{|x-y|^{n+2s}}\,dx\,dy
		\ge
		2\int_\Omega\int_{\R^n\setminus B_{2R}(0)}
		\frac{|w(x)|^2}{|x-y|^{n+2s}}\,dy\,dx.
		\]
		If $x\in\Omega$ and $y\in\R^n\setminus B_{2R}(0)$, then $|x-y|\le |x|+|y|\le R+|y|\le 3|y|/2$. Hence
		\[
		\int_{\R^n\setminus B_{2R}(0)}\frac{dy}{|x-y|^{n+2s}}\ge c(n,s,R)>0.
		\]
		It follows that the Gagliardo seminorm of $w$ controls $\|w\|_{L^2(\R^n)}$. Since the Gagliardo seminorm is equivalent to $\left(\int_{\R^n}|\xi|^{2s}|\widehat w(\xi)|^2\,d\xi\right)^{1/2}$, the estimate follows.
	\end{proof}
	
	\subsection{Well-posedness and the DN map}
	Let $F\in H^s(\R^n)$. We say that $u_F\in H^s(\R^n)$ solves the exterior Dirichlet problem with representative $F$ if
	\begin{equation}\label{eq:weak-dirichlet}
		u_F-F\in\widetilde H^s(\Omega),
		\quad
		\mathcal E_a(u_F,\phi)=0\quad\text{for all }\phi\in\widetilde H^s(\Omega).
	\end{equation}
	In other words, $u_F$ is a weak solution of
	\begin{equation}\label{eq:dirichlet-cases-section2}
		\begin{cases}
			L_a u_F=0 & \text{in }\Omega,\\
			u_F=F & \text{in }\Omega_e.
		\end{cases}
	\end{equation}
	
	\begin{proposition}\label{prop:wellposed}
		Assume \eqref{eq:ellipticity-intro}. For every $F\in H^s(\R^n)$ there is a unique $u_F\in H^s(\R^n)$ satisfying \eqref{eq:weak-dirichlet}. Moreover,
		\begin{equation}\label{eq:solution-estimate}
			\|u_F\|_{H^s(\R^n)}\le C\|F\|_{H^s(\R^n)},
		\end{equation}
		where $C$ depends only on $n$, $s$, $\Omega$ and the ellipticity constants. If $F_1-F_2\in\widetilde H^s(\Omega)$, then $u_{F_1}=u_{F_2}$.
	\end{proposition}
	
	\begin{proof}
		Write $u_F=F+w$ with $w\in\widetilde H^s(\Omega)$. The weak equation is equivalent to
		\begin{equation}\label{eq:lax-milgram-eq}
			\mathcal E_a(w,\phi)=-\mathcal E_a(F,\phi),
			\quad \phi\in\widetilde H^s(\Omega).
		\end{equation}
		The right-hand side is a bounded conjugate-linear functional on $\widetilde H^s(\Omega)$. The form $\mathcal E_a$ is bounded on $\widetilde H^s(\Omega)$ and coercive there. Indeed, by \eqref{eq:coercive-Hs-seminorm} and Lemma \ref{lem:poincare}, there is $c>0$ such that $\mathcal E_a(w,w)\ge c\|w\|_{H^s(\R^n)}^2$ for all $w\in\widetilde H^s(\Omega)$. The Lax-Milgram theorem gives a unique $w\in\widetilde H^s(\Omega)$ satisfying \eqref{eq:lax-milgram-eq}. The same estimates imply $\|w\|_{H^s(\R^n)}\le C\|F\|_{H^s(\R^n)}$, and hence \eqref{eq:solution-estimate}.
		
		It remains to check that the solution depends only on the exterior class of $F$. Suppose that $F_1-F_2\in\widetilde H^s(\Omega)$ and that $u_{F_1}$ solves \eqref{eq:weak-dirichlet} with representative $F_1$. Then $u_{F_1}-F_2=(u_{F_1}-F_1)+(F_1-F_2)\in\widetilde H^s(\Omega)$, and the weak equation is the same. By uniqueness, $u_{F_1}=u_{F_2}$.
	\end{proof}

	Let $W_1,W_2\Subset\Omega_e$ be open. We regard $f\in\widetilde H^s(W_1)$ as an element of $H^s(\R^n)$ by zero extension. Let $u_f$ denote the solution from Proposition \ref{prop:wellposed}. Define
	\begin{equation}\label{eq:DNmap-def}
		\Lambda_a^{W_1,W_2}:\widetilde H^s(W_1)\to H^{-s}(W_2)
	\end{equation}
	by
	\begin{equation}\label{eq:DN-pairing}
		\ip{\Lambda_a^{W_1,W_2}f}{g}_{H^{-s}(W_2),\widetilde H^s(W_2)}
		:=
		\mathcal E_a(u_f,g),
		\quad g\in\widetilde H^s(W_2).
	\end{equation}
	
	\begin{lemma}\label{lem:DN-bounded}
		The operator \eqref{eq:DNmap-def} is bounded. If $f\in C_c^\infty(W_1)$, then
		\begin{equation}\label{eq:DN-distributional}
			\ip{\Lambda_a^{W_1,W_2}f}{g}=\ip{L_a u_f}{g}_{\mathcal D'(W_2),C_c^\infty(W_2)},
			\quad g\in C_c^\infty(W_2).
		\end{equation}
		Consequently, equality of two restricted DN maps implies equality of the corresponding distributions $L_{a_j}u_f^{(j)}$ in $W_2$ for every common exterior datum $f\in C_c^\infty(W_1)$.
	\end{lemma}
	
	\begin{proof}
		By Lemma \ref{lem:fourier-representation} and Proposition \ref{prop:wellposed},
		\[
		|\mathcal E_a(u_f,g)|
		\le
		C\|u_f\|_{H^s(\R^n)}\|g\|_{H^s(\R^n)}
		\le
		C\|f\|_{\widetilde H^s(W_1)}\|g\|_{\widetilde H^s(W_2)}.
		\]
		This proves boundedness. Formula \eqref{eq:DN-distributional} is the definition \eqref{eq:operator-duality} of $L_a u_f$ tested against functions supported in $W_2$.
	\end{proof}
	
	\begin{lemma}[Restriction to smaller open sets]\label{lem:restriction-smaller-windows}
		Let $U_1\Subset W_1$ and $U_2\Subset W_2$ be nonempty open sets. If
		\[
		\Lambda_{a_1}^{W_1,W_2}f=\Lambda_{a_2}^{W_1,W_2}f
		\quad \text{for all }f\in\widetilde H^s(W_1),
		\]
		then
		\[
		\Lambda_{a_1}^{U_1,U_2}f=\Lambda_{a_2}^{U_1,U_2}f
		\quad \text{for all }f\in\widetilde H^s(U_1).
		\]
	\end{lemma}
	
	\begin{proof}
		The zero extension identifies $\widetilde H^s(U_1)$ with a closed subspace of $\widetilde H^s(W_1)$. Similarly, testing an element of $H^{-s}(W_2)$ against functions in $\widetilde H^s(U_2)$ gives its restriction to $U_2$. Thus, the equality on $W_1,W_2$ restricts to the equality on $U_1,U_2$.
	\end{proof}
	
	\begin{lemma}\label{lem:compact-support}
		Let $f\in C_c^\infty(W_1)$ and let $u_f$ be the solution of \eqref{eq:weak-dirichlet}. Then $\supp u_f\subset \overline\Omega\cup\supp f$.	In particular, $u_f$ is compactly supported.
	\end{lemma}
	
	\begin{proof}
		The condition $u_f-f\in\widetilde H^s(\Omega)$ implies that $u_f=f$ in $\Omega_e$ in the sense of distributions, and hence a.e. after choosing representatives. Since $f$ is supported in $W_1\subset\Omega_e$, one has $u_f=0$ in $\Omega_e\setminus\supp f$. Also, $u_f-f$ is supported in $\overline\Omega$. This proves the support inclusion. Since $\Omega$ is bounded and $f$ is compactly supported, $u_f$ is compactly supported.
	\end{proof}

	\section{Finite harmonic factorization}\label{sec:factorization}
	
	This section proves the algebraic ingredient used in the inverse problem. If $Y_\ell\in\mathscr H_\ell$, we denote by $P_\ell$ its homogeneous harmonic extension, namely $P_\ell(x)=|x|^\ell Y_\ell(x/|x|)$ for $x\ne0$. Thus, $P_\ell$ is a homogeneous harmonic polynomial of degree $\ell$.
	
	\begin{lemma}[Funk-Hecke multiplier]\label{lem:funk-hecke}
		Let $n\ge2$ and $0<s<1$. For $Y_\ell\in\mathscr H_\ell$, one has
		\begin{equation}\label{eq:FH}
			\int_{\Sn}|\xi\cdot\theta|^{2s}Y_\ell(\theta)\,d\theta
			=
			\gamma_{\ell,s}|\xi|^{2s-\ell}P_\ell(\xi),
			\quad \xi\ne0.
		\end{equation}
		The constant $\gamma_{\ell,s}$ is zero for odd $\ell$ and nonzero for every even $\ell$.
	\end{lemma}
	
	\begin{proof}
		It is enough to consider $|\xi|=1$, since both sides of \eqref{eq:FH} are homogeneous of degree $2s$ in $\xi$. For $|\xi|=1$, the function $\theta\mapsto|\xi\cdot\theta|^{2s}$ is zonal with pole $\xi$. By the Funk--Hecke formula, see for instance \cite[Theorem 1.2.9]{DaiXu2013}, there is a scalar $\gamma_{\ell,s}$ depending only on $\ell,n,s$ such that
		\[
		\int_{\Sn}|\xi\cdot\theta|^{2s}Y_\ell(\theta)\,d\theta=\gamma_{\ell,s}Y_\ell(\xi).
		\]
		More explicitly, if $\lambda=(n-2)/2$, $C_\ell^\lambda$ denotes the Gegenbauer polynomial of degree $\ell$, and $\omega_{n-1}=|\mathbb S^{n-2}|$, then
		\[
		\gamma_{\ell,s}=\omega_{n-1}\int_{-1}^{1}|t|^{2s}
		\frac{C_\ell^\lambda(t)}{C_\ell^\lambda(1)}(1-t^2)^{\lambda-1/2}\,dt,
		\]
		with the usual limiting interpretation when $n=2$. Restoring homogeneity and using
		\[
		P_\ell(\xi)=|\xi|^\ell Y_\ell(\xi/|\xi|)
		\]
		gives \eqref{eq:FH}.
		
		If $\ell$ is odd, then $\gamma_{\ell,s}=0$, since the integrand in the last formula is odd. If $\ell$ is even, the eigenvalue formula for the generalized cosine transform, see for instance \cite[Section 2]{Rubin1999}, gives
		\[
		\gamma_{\ell,s}=(-1)^{\ell/2}\frac{2\pi^{(n-1)/2}\Gamma(s+1/2)\Gamma((\ell-2s)/2)}
		{\Gamma(-s)\Gamma((\ell+n+2s)/2)}.
		\]
		Equivalently, up to a nonzero constant depending only on $n$ and $s$, $\gamma_{\ell,s}=(-1)^{\ell/2}\frac{\Gamma((\ell-2s)/2)}{\Gamma((\ell+n+2s)/2)}$.
		The Gamma function has no zeros. Since $0<s<1$, the factor $\Gamma(-s)$ is finite and nonzero, and the numerator $\Gamma((\ell-2s)/2)$ is finite and nonzero for every even $\ell\ge0$: for $\ell=0$ it is $\Gamma(-s)$, while for $\ell\ge2$ it is $\Gamma(\ell/2-s)$, which is not a pole. The denominator is also finite and nonzero. Hence, $\gamma_{\ell,s}\ne0$ for every even $\ell$.
	\end{proof}
	
	\begin{lemma}[Even cosine injectivity]\label{lem:smooth-cosine-injective}
		Let $h\in C^\infty(\Sn)$ be even. If
		\begin{equation}\label{eq:smooth-cosine-zero}
			\int_{\Sn}|\xi\cdot\theta|^{2s}h(\theta)\,d\theta=0
			\quad \text{for all }\xi\in\R^n,
		\end{equation}
		then $h=0$ on $\Sn$.
	\end{lemma}
	
	\begin{proof}
		Restricting \eqref{eq:smooth-cosine-zero} to $\xi=\omega\in\Sn$, we obtain
		\[
		T_sh(\omega):=\int_{\Sn}|\omega\cdot\theta|^{2s}h(\theta)\,d\theta=0,\quad \omega\in\Sn.
		\]
		Thus, $T_sh=0$ on $\Sn$. The operator $T_s$ is the generalized cosine transform with kernel $|\omega\cdot\theta|^{2s}$. Since this kernel is real-valued and symmetric in $\omega$ and $\theta$, the operator $T_s$ is self-adjoint on $L^2(\Sn)$.
		
		Let $\{Y_{\ell,k}\}_k$ be an orthonormal basis of $\mathscr H_\ell$. We use the $L^2(\Sn)$ inner product, which is linear in the first argument. Since $h$ is even, all of its odd spherical harmonic coefficients vanish. Indeed, if $\ell$ is odd, then $Y_{\ell,k}(-\theta)=-Y_{\ell,k}(\theta)$, and hence,
		\[
		\big(h,Y_{\ell,k}\big)_{L^2(\Sn)}=\int_{\Sn}h(\theta)\overline{Y_{\ell,k}(\theta)}\,d\theta
		=-\int_{\Sn}h(\theta)\overline{Y_{\ell,k}(\theta)}\,d\theta=0.
		\]
		Here we used the change of variables $\theta\mapsto-\theta$ and the evenness of $h$.
		
		It remains to show that the even coefficients also vanish. Let $\ell$ be even. By Lemma \ref{lem:funk-hecke}, each $Y_{\ell,k}$ is an eigenfunction of $T_s$:
		\[
		T_sY_{\ell,k}=\gamma_{\ell,s}Y_{\ell,k},
		\]
		where $\gamma_{\ell,s}\ne0$ for every even $\ell$. Since $T_sh=0$ and $T_s$ is self-adjoint,
		\[
		0=(T_sh,Y_{\ell,k})_{L^2(\Sn)}
		=(h,T_sY_{\ell,k})_{L^2(\Sn)}
		=\gamma_{\ell,s}(h,Y_{\ell,k})_{L^2(\Sn)}.
		\]
		Therefore, $(h,Y_{\ell,k})_{L^2(\Sn)}=0$ for every even $\ell$ and every $k$.
		
		We have shown that all spherical harmonic coefficients of $h$ vanish, both in odd and even degrees. Since spherical harmonics are complete in $L^2(\Sn)$, it follows that $h=0$ in $L^2(\Sn)$. Finally, $h$ is smooth, hence continuous, so $h=0$ pointwise on $\Sn$.
	\end{proof}
	
	\begin{lemma}[Finite cosine injectivity]\label{lem:cosine-injective}
		Let $a\in\mathcal A_m$. If
		\[
		\int_{\Sn}|\xi\cdot\theta|^{2s}a(\theta)\,d\theta=0
		\quad\text{for all }\xi\in\R^n,
		\]
		then $a=0$.
	\end{lemma}
	
	\begin{proof}
		This follows immediately from Lemma \ref{lem:smooth-cosine-injective}, since every $a\in\mathcal A_m$ is smooth and even. In other words, one may write $a=\sum a_{\ell,k}Y_{\ell,k}$ over even $0\le \ell\le 2m$ and use Lemma \ref{lem:funk-hecke} together with the nonvanishing of the even multipliers $\gamma_{\ell,s}$.
	\end{proof}
	
	\begin{lemma}[Finite harmonic factorization]\label{lem:factorization}
		Let $a\in\mathcal A_m$. Then there exists a homogeneous polynomial $Q_a$, of degree $2m$ unless $a\equiv0$, such that
		\begin{equation}\label{eq:factorization-symbol}
			|\xi|^{2m}A_a(\xi)=|\xi|^{2s}Q_a(\xi),
			\quad \xi\in\R^n.
		\end{equation}
		Equivalently,
		\begin{equation}\label{eq:factorization-operator}
			(-\Delta)^mL_a=(-\Delta)^sQ_a(D)
		\end{equation}
		as Fourier multipliers. Moreover, the map $a\mapsto Q_a$ is injective on $\mathcal A_m$.
	\end{lemma}
	
	\begin{proof}
		Write $a=\sum_{\ell,k}a_{\ell,k}Y_{\ell,k}$, where the sum is over even $0\le\ell\le2m$ and $\{Y_{\ell,k}\}_k$ is an orthonormal basis of $\mathscr H_\ell$. Let $P_{\ell,k}$ be the homogeneous harmonic extension of $Y_{\ell,k}$. By Lemma \ref{lem:funk-hecke}, for $\xi\ne0$,
		\[
		A_a(\xi)
		=
		\sum_{\substack{0\le\ell\le2m\\ \ell\text{ even}}}\sum_k
		c_{n,s}a_{\ell,k}\gamma_{\ell,s}|\xi|^{2s-\ell}P_{\ell,k}(\xi).
		\]
		Multiplying by $|\xi|^{2m}$ gives
		\[
		|\xi|^{2m}A_a(\xi)
		=
		|\xi|^{2s}
		\sum_{\substack{0\le\ell\le2m\\ \ell\text{ even}}}\sum_k
		c_{n,s}a_{\ell,k}\gamma_{\ell,s}|\xi|^{2m-\ell}P_{\ell,k}(\xi).
		\]
		Since $\ell$ is even, $|\xi|^{2m-\ell}=(|\xi|^2)^{m-\ell/2}$ is a polynomial. We define
		\begin{equation}\label{eq:Qa-definition}
			Q_a(\xi)
			:=
			\sum_{\substack{0\le\ell\le2m\\ \ell\text{ even}}}\sum_k
			c_{n,s}a_{\ell,k}\gamma_{\ell,s}|\xi|^{2m-\ell}P_{\ell,k}(\xi).
		\end{equation}
		Then $Q_a$ is homogeneous of degree $2m$ unless $a\equiv0$, in which case $Q_a\equiv0$. The identity \eqref{eq:factorization-symbol} follows for $\xi\ne0$. Since both sides are locally integrable homogeneous functions and both vanish at $\xi=0$ in the pointwise representative, the identity holds as an equality of Fourier multipliers on $\R^n$.
		
		The operator identity \eqref{eq:factorization-operator} follows from \eqref{eq:factorization-symbol}: the multiplier of $(-\Delta)^mL_a$ is $|\xi|^{2m}A_a(\xi)$, while the multiplier of $(-\Delta)^sQ_a(D)$ is $|\xi|^{2s}Q_a(\xi)$. Finally, if $Q_a=0$, then $A_a=0$ for all $\xi\ne0$. By Lemma \ref{lem:cosine-injective}, this implies $a=0$. Hence $a\mapsto Q_a$ is injective on $\mathcal A_m$.
	\end{proof}
	
	\begin{lemma}[Factorization on Sobolev distributions]\label{lem:factorization-distributions}
		Let $a\in\mathcal A_m$ and let $Q_a$ be as in Lemma \ref{lem:factorization}. If $u\in H^s(\R^n)$, then
		\begin{equation}\label{eq:factorization-distribution}
			(-\Delta)^mL_a u=(-\Delta)^sQ_a(D)u
		\end{equation}
		in $H^{-s-2m}(\R^n)$ and hence in $\mathcal D'(\R^n)$.
	\end{lemma}
	
	\begin{proof}
		By Lemma \ref{lem:fourier-representation}, $L_a u$ is the Fourier multiplier with symbol $A_a$. Since $|A_a(\xi)|\le C|\xi|^{2s}$ and $u\in H^s(\R^n)$, one has $L_a u\in H^{-s}(\R^n)$. Applying the local differential operator $(-\Delta)^m$ gives a distribution in $H^{-s-2m}(\R^n)$ whose Fourier transform is $|\xi|^{2m}A_a(\xi)\widehat u(\xi)$.
		
		On the other hand, $Q_a(D)u\in H^{s-2m}(\R^n)$, and $(-\Delta)^sQ_a(D)u\in H^{-s-2m}(\R^n)$ has Fourier transform $|\xi|^{2s}Q_a(\xi)\widehat u(\xi)$. By Lemma \ref{lem:factorization}, the two multipliers $|\xi|^{2m}A_a(\xi)$ and $|\xi|^{2s}Q_a(\xi)$ are the same locally integrable function with polynomial growth. Hence, the two Fourier transforms agree in $\mathcal S'(\R^n)$, which proves \eqref{eq:factorization-distribution}.
	\end{proof}
	
	\begin{lemma}[Identification of constant coefficient operators]\label{lem:local-operator-id}
		Let $P(D)=\sum_{|\alpha|\le M}c_\alpha D^\alpha$ be a constant coefficient differential operator with $D=-\mathsf i\nabla$, and let $P(\xi)=\sum_{|\alpha|\le M}c_\alpha \xi^\alpha$	be its polynomial symbol. Let $O\subset\R^n$ be nonempty and open. If $P(D)f=0$ in $O$, for every $f\in C_c^\infty(O)$, then $P(\xi)\equiv0$.
	\end{lemma}
	
	\begin{proof}
		Fix $x_0\in O$. Choose $\chi\in C_c^\infty(O)$ such that $\chi=1$ in a neighborhood $V\Subset O$ of $x_0$. For each $\eta\in\R^n$, define the test function $\varphi_\eta(x)=\chi(x)e^{\mathsf i x\cdot\eta}$, then $\varphi_\eta\in C_c^\infty(O)$. The hypothesis gives $P(D)\varphi_\eta=0\quad \text{in }O$. On the neighborhood $V$, the cutoff $\chi$ is identically one. Hence, by locality of $P(D)$,
		\[
		P(D)\varphi_\eta=P(D)e^{\mathsf i x\cdot\eta}=P(\eta)e^{\mathsf i x\cdot\eta}\quad \text{in }V.
		\]
		Therefore, $P(\eta)e^{\mathsf i x\cdot\eta}=0$ in $V$. Since $e^{\mathsf i x\cdot\eta}$ never vanishes, we get $P(\eta)=0$. The vector $\eta\in\R^n$ was arbitrary, so $P(\xi)=0$ for all $\xi\in\R^n$. Thus $P$ is the zero polynomial.
	\end{proof}

	\section{Off-diagonal kernels and analytic continuation}\label{sec:analytic-kernel}
	
	This section proves the kernel representation and analytic continuation facts used for Theorem \ref{thm:analytic-main}. Throughout this section, $a$ denotes a smooth elliptic angular density. Let $R_a:H^{-s}(\Omega)\to\widetilde H^s(\Omega)$ be the solution operator for the zero exterior problem: for $F\in H^{-s}(\Omega)$, $R_aF$ is the unique element of $\widetilde H^s(\Omega)$ satisfying $\mathcal E_a(R_aF,\phi)=\ip{F}{\phi}$ for all $\phi\in\widetilde H^s(\Omega)$. The existence and boundedness of $R_a$ follow from the same Lax-Milgram argument as Proposition \ref{prop:wellposed}.
	
	When $v\in\widetilde H^s(\Omega)$ and $F\in H^{-s}(\Omega)$, we use the reverse duality pairing
	\[
	\ip{v}{F}_{\widetilde H^s(\Omega)\times H^{-s}(\Omega)}
	:=
	\overline{\ip{F}{v}_{H^{-s}(\Omega)\times \widetilde H^s(\Omega)}}.
	\]
	Thus, this pairing is linear in $v$ and conjugate-linear in $F$.
	
	For $x\in\Omega_e$, set $k_x^a(z):=K_a(x-z)$ for $z\in\Omega$. Locally in $x\in\Omega_e$, the point $x$ stays a positive distance away from $\Omega$, and therefore $k_x^a\in C^\infty(\overline\Omega)$ in the interior variable. In particular $k_x^a\in L^2(\Omega)\subset H^{-s}(\Omega)$.
	
	\begin{lemma}[Off-diagonal DN kernel]\label{lem:off-diagonal-kernel}
		Let $U,V\Subset\Omega_e$ be nonempty open sets with $\overline U\cap\overline V=\emptyset$. Then $\Lambda_a^{U,V}$ has a distribution kernel represented by
		\begin{equation}\label{eq:Sa-kernel}
			S_a(x,y)=-K_a(x-y)-\ip{R_a k_y^a}{k_x^a}_{\widetilde H^s(\Omega)\times H^{-s}(\Omega)},\quad x\in V, 	\quad y\in U.
		\end{equation}
		That is, for $f\in C_c^\infty(U)$ and $g\in C_c^\infty(V)$,
		\begin{equation}\label{eq:kernel-pairing}
			\big\langle\Lambda_a^{U,V}f, g\big\rangle=\int_V\int_U S_a(x,y)f(y)\overline{g(x)}\,dy\,dx.
		\end{equation}
		Moreover, if $x,y\in\Omega_e$ and $x\ne y$, then \eqref{eq:Sa-kernel} defines the off-diagonal kernel of the exterior DN map in a neighborhood of $(x,y)$.
	\end{lemma}
	
	\begin{proof}
		Let $f\in C_c^\infty(U)$ and write $u_f=f+v_f$, where $v_f\in\widetilde H^s(\Omega)$. We first identify $v_f$. If $\phi\in C_c^\infty(\Omega)$, then the supports of $f$ and $\phi$ are disjoint. Using the definition of the energy form and the evenness of $K_a$, one obtains
		\begin{equation}\label{eq:cross-term-f-phi}
			\mathcal E_a(f,\phi)=-\int_\Omega\bigg(\int_U K_a(z-y)f(y)\,dy\bigg)\overline{\phi(z)}\,dz.
		\end{equation}
		The same identity holds for all $\phi\in\widetilde H^s(\Omega)$ by density. Since $u_f$ solves the equation in $\Omega$, $\mathcal E_a(v_f,\phi)=-\mathcal E_a(f,\phi)$. The function
		\[
		F_f(z):=\int_U K_a(z-y)f(y)\,dy,\quad z\in\Omega,
		\]
		is smooth on $\overline\Omega$, since $\overline U$ and $\overline\Omega$ are disjoint. In particular $F_f\in L^2(\Omega)\subset H^{-s}(\Omega)$. The weak equation for $v_f$ is 
		\[
		\mathcal E_a(v_f,\phi)=\ip{F_f}{\phi}_{H^{-s}(\Omega)\times\widetilde H^s(\Omega)},
		\quad \phi\in\widetilde H^s(\Omega).
		\]
		By the definition of $R_a$, this gives $v_f=R_aF_f$. Also, by the evenness of $K_a$, $F_f=\int_U f(y)k_y^a\,dy \in H^{-s}(\Omega)$, meaning that, for every $\phi\in\widetilde H^s(\Omega)$, 
		\[
		\ip{F_f}{\phi}_{H^{-s}(\Omega),\widetilde H^s(\Omega)}
		=
		\int_U f(y)\ip{k_y^a}{\phi}_{H^{-s}(\Omega)\times\widetilde H^s(\Omega)}\,dy.
		\]
		
		Let $g\in C_c^\infty(V)$. Since $U$ and $V$ are disjoint and both are contained in $\Omega_e$, the direct exterior-to-exterior contribution is
		\begin{equation}\label{eq:direct-exterior-term}
			\mathcal E_a(f,g)=-\int_V\int_U K_a(x-y)f(y)\overline{g(x)}\,dy\,dx.
		\end{equation}
		Similarly, since $v_f$ is supported in $\overline\Omega$ and $g$ is supported in $V\Subset\Omega_e$,
		\begin{equation}\label{eq:correction-exterior-term}
			\mathcal E_a(v_f,g)=-\int_V\ip{v_f}{k_x^a}_{\widetilde H^s(\Omega)\times H^{-s}(\Omega)}\overline{g(x)}\,dx.
		\end{equation}
		Using the identity
		\[
		v_f=R_aF_f=\int_U f(y)R_a k_y^a\,dy\quad \text{in }\widetilde H^s(\Omega),
		\]
		which follows from the bounded linearity of $R_a$, we obtain from \eqref{eq:correction-exterior-term}
		\[
		\mathcal E_a(v_f,g)
		=
		-\int_V\int_U \ip{R_a k_y^a}{k_x^a}_{\widetilde H^s(\Omega)\times H^{-s}(\Omega)}f(y)\overline{g(x)}\,dy\,dx.
		\]
		Adding this identity to \eqref{eq:direct-exterior-term} proves \eqref{eq:Sa-kernel} and \eqref{eq:kernel-pairing}. The argument is local in $(x,y)$ as long as $x$ and $y$ remain in $\Omega_e$ and away from the diagonal, which proves the last assertion.
	\end{proof}
	
	\begin{lemma}[Analyticity of the off-diagonal kernel]\label{lem:analytic-kernel}
		Assume $a\in C^\omega(\Sn)$. Let $G$ be a connected component of $\Omega_e$. Then $S_a(x,y)$ is real analytic in $(x,y)$ on
		\begin{equation}\label{eq:pair-domain}
			\mathcal D_G:=\{(x,y)\in G\times G:x\ne y\}.
		\end{equation}
	\end{lemma}
	
	\begin{proof}
		Since $a$ is real analytic on $\Sn$, it admits local real analytic extensions in neighborhoods of points of $\Sn$. The map $z\mapsto z/|z|$ is real analytic on $\R^n\setminus\{0\}$, and so is $z\mapsto |z|^{-n-2s}$. Hence, by working locally and using compactness on the relevant sets, the homogeneous kernel $K_a(z)=\frac{a(z/|z|)}{|z|^{n+2s}}$ is real analytic for $z\in\R^n\setminus\{0\}$. In particular, $K_a(x-y)$ is real analytic on the set $x\ne y$.
		
		It remains to consider the correction term in \eqref{eq:Sa-kernel}. Fix $(x_0,y_0)\in\mathcal D_G$. Choose open neighborhoods $X,Y\Subset G$ of $x_0,y_0$ such that $\overline X\cap\overline Y=\emptyset$, $\dist(X,\Omega)>0$ and $\dist(Y,\Omega)>0$. For $x\in X$ and $z\in\Omega$, the set of points $x-z$ stays in a compact subset of $\R^n\setminus\{0\}$. Standard Cauchy estimates for real analytic functions, see for instance \cite[Chapter 1]{KrantzParks2002}, imply that for every compact $X_0\Subset X$ there are constants $C_X,R_X>0$ such that
		\[
		|\partial_x^\alpha K_a(x-z)|\le C_XR_X^{|\alpha|}\alpha!
		\]
		for all $x\in X_0$, $z\in\Omega$ and multiindices $\alpha$. Since $\Omega$ is bounded, the same estimate holds in $L^2(\Omega)$:
		\[
		\|\partial_x^\alpha k_x^a\|_{L^2(\Omega)}
		\le C_XR_X^{|\alpha|}\alpha!,
		\quad x\in X_0.
		\]
		Hence, the Taylor series of $x\mapsto k_x^a$ converges in $L^2(\Omega)$ locally uniformly in $x$. Moreover, $x\mapsto k_x^a$ is an $L^2(\Omega)$-valued, and then $H^{-s}(\Omega)$-valued, real analytic map. The same argument applies to $y\mapsto k_y^a$.
		
		The operator $R_a:H^{-s}(\Omega)\to\widetilde H^s(\Omega)$ is bounded, and the duality pairing between $\widetilde H^s(\Omega)$ and $H^{-s}(\Omega)$ is continuous. Therefore $(x,y)\mapsto\ip{R_a k_y^a}{k_x^a}$ is a real analytic function on $X\times Y$. Since $(x_0,y_0)$ was arbitrary, $S_a$ is real analytic on $\mathcal D_G$.
	\end{proof}
	
	\begin{lemma}[Far-field asymptotics]\label{lem:far-field-asymptotic}
		Let $a$ be a smooth elliptic angular density. There exists a constant $C>0$ such that for all $x,y\in\Omega_e$ with $x\ne y$,
		\begin{equation}\label{eq:far-field-bound}
			\abs{S_a(x,y)+K_a(x-y)}\le \frac{C}{\dist(x,\Omega)^{n+2s}\dist(y,\Omega)^{n+2s}}.
		\end{equation}
	\end{lemma}
	
	\begin{proof}
		By the off-diagonal kernel representation \eqref{eq:Sa-kernel}, one has
		\[
		S_a(x,y)+K_a(x-y)=-\ip{R_a k_y^a}{k_x^a}_{\widetilde H^s(\Omega)\times H^{-s}(\Omega)}.
		\]
		Thus, it remains to estimate the correction term on the right-hand side.
		
		Since $a$ is smooth on $\Sn$, it is bounded. If $x\in\Omega_e$ and $z\in\Omega$, then $|x-z|\ge \dist(x,\Omega)$. Hence,
		\[
		|k_x^a(z)|=|K_a(x-z)| \le \|a\|_{L^\infty(\Sn)}\dist(x,\Omega)^{-n-2s},\quad z\in\Omega.
		\]
		Since $\Omega$ is bounded, this gives $\|k_x^a\|_{L^2(\Omega)} \le C\dist(x,\Omega)^{-n-2s}$. 
		The continuous embedding $L^2(\Omega)\hookrightarrow H^{-s}(\Omega)$ implies
		\[
		\|k_x^a\|_{H^{-s}(\Omega)} \le C\dist(x,\Omega)^{-n-2s}.
		\]
		The same estimate holds with $y$ in place of $x$.
		
		By the boundedness of the solution operator $R_a:H^{-s}(\Omega)\to\widetilde H^s(\Omega)$,
		\[
		\|R_a k_y^a\|_{\widetilde H^s(\Omega)}
		\le
		C\|k_y^a\|_{H^{-s}(\Omega)}.
		\]
		Therefore, by duality,
		\[
		\big| \ip{R_a k_y^a}{k_x^a}_{\widetilde H^s(\Omega)\times H^{-s}(\Omega)}\big| \le C\|k_y^a\|_{H^{-s}(\Omega)}\|k_x^a\|_{H^{-s}(\Omega)}.
		\]
		Combining the previous estimates gives
		\[
		\abs{S_a(x,y)+K_a(x-y)}
		\le
		\frac{C}{\dist(x,\Omega)^{n+2s}\dist(y,\Omega)^{n+2s}},
		\]
		which proves \eqref{eq:far-field-bound}.
	\end{proof}
	
	\begin{lemma}[Connectedness of the exterior pair domain]\label{lem:pair-domain-connected}
		Let $n\ge2$ and let $G\subset\R^n$ be a nonempty connected open set. Then $\mathcal D_G$ is path connected, where $\mathcal{D}_G$ is defined by \eqref{eq:pair-domain}. In particular, $\mathcal D_G$ is connected.
	\end{lemma}
	
	\begin{proof}
		We first record a basic topological fact. If $p\in G$, then $G\setminus\{p\}$ is path connected. Indeed, since $G$ is open and connected in $\R^n$, it is path-connected. Let $a,b\in G\setminus\{p\}$. Choose a continuous path in $G$ joining $a$ to $b$. By compactness of its image, one can cover the path by finitely many open balls $B_0,\ldots,B_N$ compactly contained in $G$, such that $a\in B_0$, $b\in B_N$ and $B_{i-1}\cap B_i\ne\emptyset$ for $1\le i\le N$. Choose points
		\[
		c_i\in (B_{i-1}\cap B_i)\setminus\{p\},	\quad 1\le i\le N, \quad c_0=a, \quad \text{and}\quad   c_{N+1}=b.
		\]
		This is possible because each $B_{i-1}\cap B_i$ is a nonempty open set. For each $0\le i\le N$, the points $c_i$ and $c_{i+1}$ belong to $B_i\setminus\{p\}$. Since $n\ge2$, a punctured ball in $\R^n$ is path connected. Thus, $c_i$ and $c_{i+1}$ can be joined by a path in $B_i\setminus\{p\}$. Concatenating these finitely many paths gives a path in $G\setminus\{p\}$ from $a$ to $b$.
		
		We now prove the assertion for $\mathcal D_G$. Let $(x_0,y_0),(x_1,y_1)\in\mathcal D_G$. Choose two distinct points $p,q\in G\setminus\{x_0,y_0,x_1,y_1\}$. This is possible since $G$ contains a nonempty open ball. By the previous discussions, there is a path $\alpha$ in $G\setminus\{y_0\}$ from $x_0$ to $p$. Hence, $t\mapsto(\alpha(t),y_0)$ is a path in $\mathcal D_G$ from $(x_0,y_0)$ to $(p,y_0)$. Next, there is a path $\beta$ in $G\setminus\{p\}$ from $y_0$ to $q$, and therefore $t\mapsto(p,\beta(t))$ is a path in $\mathcal D_G$ from $(p,y_0)$ to $(p,q)$. Similarly, there is a path $\gamma$ in $G\setminus\{q\}$ from $p$ to $x_1$, which gives a path from $(p,q)$ to $(x_1,q)$, and there is a path $\delta$ in $G\setminus\{x_1\}$ from $q$ to $y_1$, which gives a path from $(x_1,q)$ to $(x_1,y_1)$. Concatenating these four paths gives a path in $\mathcal D_G$ from $(x_0,y_0)$ to $(x_1,y_1)$. Hence $\mathcal D_G$ is path connected.
	\end{proof}
	
	\begin{lemma}[Analytic continuation to the far field]\label{lem:near-to-far}
		Let $a_1,a_2\in C^\omega(\Sn)$ be real analytic elliptic angular densities, let $G$ be the unbounded connected component of $\Omega_e$, and set $S_{12}:=S_{a_1}-S_{a_2}$. If $S_{12}=0$ on a nonempty open subset of $\mathcal D_G$, then $S_{12}=0$ on $\mathcal D_G$.
	\end{lemma}
	
	\begin{proof}
		By Lemma \ref{lem:analytic-kernel}, $S_{12}$ is real analytic on $\mathcal D_G$. By Lemma \ref{lem:pair-domain-connected}, $\mathcal D_G$ is connected. The conclusion follows from the identity theorem for real analytic functions, see for instance \cite[Chapter 1]{KrantzParks2002}.
	\end{proof}

	\section{Inverse problems}\label{sec:further}
	
	This section proves the main results. We first record the Alessandrini identity associated with changing the angular density. This identity is included to display the bilinear information carried by equality of the exterior DN maps and to indicate the obstruction in the general smooth separated case. The proofs of Theorem \ref{thm:overlap}, Theorem \ref{thm:separated-main}, and Theorem \ref{thm:analytic-main} then use the three mechanisms described above: the exterior diagonal singularity, the finite harmonic factorization, and analytic continuation of the off-diagonal kernel.
	
	\subsection{Alessandrini's identity and structural remarks}
	
	\begin{proposition}[Alessandrini identity]\label{prop:alessandrini-leading}
		Let $a_1,a_2\in L^\infty(\Sn)$ be real-valued even angular densities whose symbols satisfy the ellipticity condition \eqref{eq:ellipticity-intro}. Let $W_1,W_2\Subset\Omega_e$ be nonempty open sets. Assume that	
		\[	
		\Lambda_{a_1}^{W_1,W_2}f=\Lambda_{a_2}^{W_1,W_2}f \quad \text{for all }f\in C^\infty_c(W_1).
		\] 
		Let $f\in C_c^\infty(W_1)$ and $g\in C_c^\infty(W_2)$. For $j=1,2$, let $u_f^{(j)}$ and $u_g^{(j)}$ denote the solutions with exterior data $f$ and $g$, respectively, for the operator $L_{a_j}$. Then
		\begin{equation}\label{eq:leading-alessandrini}
			\mathcal E_{a_1-a_2}(u_f^{(1)},u_g^{(2)})=0,
		\end{equation}
		where $\mathcal{E}_a(u,v)$ is defined by \eqref{eq:energy-def}.
		Equivalently,
		\begin{equation}\label{eq:leading-alessandrini-expanded}
			\frac12\iint_{\R^n\times\R^n}
			\delta u_f^{(1)}(x,y)\overline{\delta u_g^{(2)}(x,y)}
			\frac{(a_1-a_2)((x-y)/|x-y|)}{|x-y|^{n+2s}}\,dx\,dy=0,
		\end{equation}
		where $\delta u(x,y)=u(x)-u(y)$.
	\end{proposition}
	
	\begin{proof}
		We use the weak formulation and the definition of the restricted DN map. Since $u_g^{(2)}$ has exterior value $g$, one has $u_g^{(2)}-g\in\widetilde H^s(\Omega)$. As $u_f^{(1)}$ solves the equation with coefficient $a_1$ in $\Omega$, $\mathcal E_{a_1}(u_f^{(1)},u_g^{(2)}-g)=0$. Hence
		\begin{equation}\label{eq:aless-step1}
			\mathcal E_{a_1}(u_f^{(1)},u_g^{(2)})
			=\mathcal E_{a_1}(u_f^{(1)},g)
			=\ip{\Lambda_{a_1}^{W_1,W_2}f}{g}.
		\end{equation}
		Next, $u_f^{(1)}-f\in\widetilde H^s(\Omega)$. Since $u_g^{(2)}$ solves the equation with coefficient $a_2$ in $\Omega$, $\mathcal E_{a_2}(u_g^{(2)},u_f^{(1)}-f)=0$. By Hermitian symmetry, $\mathcal E_{a_2}(u_f^{(1)}-f,u_g^{(2)})=0$. Thus
		\begin{equation}\label{eq:aless-step2}
			\mathcal E_{a_2}(u_f^{(1)},u_g^{(2)})
			=\mathcal E_{a_2}(f,u_g^{(2)}).
		\end{equation}
		We identify the right-hand side of \eqref{eq:aless-step2} with the DN pairing for $a_2$. Let $u_f^{(2)}$ be the solution with coefficient $a_2$ and exterior datum $f$. Since $u_f^{(2)}-f\in\widetilde H^s(\Omega)$ and $u_g^{(2)}$ solves the $a_2$-equation in $\Omega$, Hermitian symmetry gives $\mathcal E_{a_2}(u_f^{(2)}-f,u_g^{(2)})=0$. Therefore $\mathcal E_{a_2}(f,u_g^{(2)})=\mathcal E_{a_2}(u_f^{(2)},u_g^{(2)})$. Also, since $u_g^{(2)}-g\in\widetilde H^s(\Omega)$ and $u_f^{(2)}$ solves the $a_2$-equation in $\Omega$, one has
		\[
		\mathcal E_{a_2}(u_f^{(2)},u_g^{(2)})
		=\mathcal E_{a_2}(u_f^{(2)},g)
		=\ip{\Lambda_{a_2}^{W_1,W_2}f}{g}.
		\]
		Combining this with \eqref{eq:aless-step2}, we get
		\begin{equation}\label{eq:aless-step3}
			\mathcal E_{a_2}(u_f^{(1)},u_g^{(2)})
			=\ip{\Lambda_{a_2}^{W_1,W_2}f}{g}.
		\end{equation}
		The equality of the DN maps and \eqref{eq:aless-step1}, \eqref{eq:aless-step3} imply $\mathcal E_{a_1}(u_f^{(1)},u_g^{(2)})-\mathcal E_{a_2}(u_f^{(1)},u_g^{(2)})=0$. This is \eqref{eq:leading-alessandrini}. The expanded formula \eqref{eq:leading-alessandrini-expanded} follows from the definition of the energy forms.
	\end{proof}
	
	\para{The overlapping regime}
	Theorem \ref{thm:overlap} shows that overlapping exterior measurements determine smooth elliptic angular densities without any finite-dimensional assumption. The mechanism is local in the exterior variables. If $U\Subset W_1\cap W_2$ and $f\in C_c^\infty(U)$, then the solution satisfies $u_f=f$ in $U$. Hence $\Lambda_a^{U,U}$ contains the same leading diagonal singularity as $L_a$ on $U$. This diagonal singularity determines $K_a(z)=a(z/|z|)|z|^{-n-2s}$ and the symbol $A_a(\xi)=c_{n,s}\int_{\Sn}|\xi\cdot\theta|^{2s}a(\theta)\,d\theta$. The injectivity of the even cosine transform then gives $a$.
	
	The separated regime \eqref{eq:separated-intro} is different. If $x\in W_2$ and $y\in W_1$, then $x$ and $y$ stay away from the diagonal. Thus $K_a(x-y)$ is smooth on $W_2\times W_1$, and the exterior diagonal singularity is not measured. This is the reason that the separated problem requires a different mechanism.
	
	\para{Finite harmonic and analytic mechanisms}
	The proof of Theorem \ref{thm:separated-main} relies on the factorization $(-\Delta)^mL_a=(-\Delta)^sQ_a(D)$, where $Q_a(D)$ is a local constant coefficient differential operator. The locality of $Q_a(D)$ is decisive. If $u=0$ in an open set, then $Q_a(D)u=0$ in the same open set. This permits one to convert the separated exterior identity into the standard strong unique continuation property for $(-\Delta)^s$.
	
	This argument is specific to finite even spherical harmonic densities. The analytic result uses a different replacement for locality: the off-diagonal DN kernel is real analytic in the exterior variables. Equality on $W_2\times W_1$ therefore propagates by analytic continuation through the exterior pair domain, and the far-field asymptotic isolates the direct stable kernel.
	
	\para{General angular densities}
	For arbitrary smooth angular densities, neither mechanism is directly available. The spherical harmonic expansion is infinite, and multiplication by a finite power of $|\xi|^2$ no longer turns $A_a(\xi)$ into $|\xi|^{2s}$ times a polynomial. The operator replacing $Q_a(D)$ would be of infinite order and would not preserve vanishing on open sets. At the same time, a smooth off-diagonal DN kernel does not admit analytic continuation.
	
	The identity \eqref{eq:leading-alessandrini-expanded} involves jump products, and the unknown coefficient depends on the jump direction $(x-y)/|x-y|$. Thus, the information in the identity is not a pointwise product inside $\Omega$, but an angular family of nonlocal products over pairs of points. In the general smooth separated case, one would need a way to control these angular jump products without access to the exterior diagonal singularity.
	
	A related formal viewpoint is obtained by decomposing the bilinear form into exterior and interior variables. Let $E=\Omega_e$. The notation below is only meant to indicate the four pieces of the operator associated with $\mathcal E_a$ with respect to the decomposition $E\cup\Omega$. The first subscript denotes where the expression is tested, and the second subscript denotes where the input is supported. Thus, for $f,g\in C_c^\infty(E)$ and $v,\phi\in C_c^\infty(\Omega)$, define
	\[
	\ip{A^a_{EE}f}{g}:=\mathcal E_a(f,g),\quad \ip{A^a_{E\Omega}v}{g}:=\mathcal E_a(v,g),\quad 
	\ip{A^a_{\Omega E}f}{\phi}:=\mathcal E_a(f,\phi),\text{ and } \ip{A^a_{\Omega\Omega}v}{\phi}:=\mathcal E_a(v,\phi).
	\]
	Thus, $A^a_{EE}$ is the exterior-to-exterior part, $A^a_{\Omega E}$ is the part by which exterior data enter the weak equation in $\Omega$, $A^a_{E\Omega}$ is the part by which an interior correction contributes to exterior observations, and $A^a_{\Omega\Omega}$ is the interior Dirichlet part of the operator. Formally, this is the block decomposition
	\[
	A^a=\begin{pmatrix} A^a_{EE} & A^a_{E\Omega} \\ A^a_{\Omega E} & A^a_{\Omega\Omega} \end{pmatrix}.
	\]
	
	Let $u=f+v$, where $f$ is the exterior datum and $v\in\widetilde H^s(\Omega)$ is the interior correction. The weak equation $L_a u=0$ in $\Omega$ means $\mathcal E_a(f+v,\phi)=0$, for all $\phi\in\widetilde H^s(\Omega)$. Equivalently,
	\[
	\ip{A^a_{\Omega E}f}{\phi}+\ip{A^a_{\Omega\Omega}v}{\phi}=0\quad \text{for all }\phi\in\widetilde H^s(\Omega).
	\]
	In the formal block notation, this is $A^a_{\Omega E}f+A^a_{\Omega\Omega}v=0$. Solving this interior weak equation gives
	\[
	v=-(A^a_{\Omega\Omega})^{-1}A^a_{\Omega E}f.
	\]
	Substituting this into the exterior pairing gives
	\[
	\Lambda_a f=A^a_{EE}f+A^a_{E\Omega}v=\big(A^a_{EE}-A^a_{E\Omega}(A^a_{\Omega\Omega})^{-1}A^a_{\Omega E}\big)f.
	\]
	This is the Schur complement viewpoint. The direct term $A^a_{EE}$ is the exterior-to-exterior interaction. The second term describes the effect of exterior data entering $\Omega$, producing an interior correction, and then contributing back to the exterior observation.
	
	One can see the same structure in the off-diagonal kernel formula. If $x,y\in E$ and $x\ne y$, then the direct exterior-to-exterior part contributes $-K_a(x-y)$. The correction term is obtained by sending the point source at $y$ into $\Omega$ through $k_y^a(z)=K_a(y-z)$, solving the interior problem by the zero exterior solution operator $R_a$, and observing from $x$ through $k_x^a(z)=K_a(x-z)$. Thus
	\[
	S_a(x,y)=-K_a(x-y)-\ip{R_a k_y^a}{k_x^a}_{\widetilde H^s(\Omega)\times H^{-s}(\Omega)}.
	\]
	For $h=a_1-a_2$, the direct exterior-to-exterior part of the difference is
	\[
	(\mathcal P h)(x,y)=-\frac{h\big((x-y)/|x-y|\big)}{|x-y|^{n+2s}},\quad x\in W_2,\quad y\in W_1.
	\]
	Thus, at the level of the direct term, the measured directions are those contained in $W_2-W_1$. If these directions cover the sphere up to evenness, then this direct term contains strong angular information about $h$. The difficulty is that the measured DN map also contains the interior correction term above, which depends on the full interior solution operator. Controlling this correction is precisely the obstruction in the general smooth separated case.
	
	A treatment of arbitrary smooth angular densities in the fixed separated case would require a way to control this correction, or a propagation result for the Schur kernel. Such an argument would be analogous in spirit to entanglement principles for different fractional powers \cite{FL24,FKU24}, but it is more delicate here because all angular components have the same homogeneity.
	
	\subsection{Proofs of main results}\label{sec:proofs}
	
	We use the following strong unique continuation property for the fractional Laplacian. It is the only external unique continuation result used in the proofs.
	
	\begin{theorem}[Fractional unique continuation]\label{thm:fractional-UCP}
		Let $0<s<1$, let $r\in\R$ and let $O\subset\R^n$ be nonempty and open. If $w\in H^r(\R^n)$ satisfies
		\[
		w=(-\Delta)^s w=0\quad\text{in }O, 
		\]
		then $w=0$ in $\R^n$.
	\end{theorem}
	
	This formulation follows from the well-known fractional unique continuation results of \cite[Theorem 1.2]{GSU20}. We use it only for distributions of the form $Q(D)u$, where $u\in H^s(\R^n)$ and $Q$ is a polynomial.
	
	\begin{lemma}[Off-diagonal smoothing]\label{lem:off-diagonal-smoothing}
		Let $a\in C^\infty(\Sn)$ be a smooth elliptic angular density and let $U\Subset\Omega_e$ be nonempty and open. For $f\in C_c^\infty(U)$, let $u_f$ solve \eqref{eq:weak-dirichlet} and define $R_a^Uf:=\Lambda_a^{U,U}f-L_af|_U$, where $L_a$ is defined by \eqref{eq:operator-intro}. Let $K_0\Subset U$, let $\eta\in C_c^\infty(U)$, let $\xi\in\R^n\setminus\{0\}$, and set $f_h(x)=\eta(x)e^{\mathsf ix\cdot\xi/h}$ for $0<h<1$. Then for every integer $N\ge0$ there is a constant $C_N>0$ such that
		\begin{equation}\label{eq:offdiag-smoothing-estimate}
			\|R_a^Uf_h\|_{L^2(K_0)}\le C_Nh^N.
		\end{equation}
	\end{lemma}
	
	\begin{proof}
		Write $u_f=f+w_f$ with $w_f\in\widetilde H^s(\Omega)$. Then $w_f$ is the unique solution of
		\begin{equation}\label{eq:wf-equation-overlap}
			\mathcal E_a(w_f,\phi)=-\mathcal E_a(f,\phi),
			\quad \phi\in\widetilde H^s(\Omega).
		\end{equation}
		Since $U\Subset\Omega_e$, the sets $U$ and $\Omega$ have positive distance. For $\phi\in C_c^\infty(\Omega)$ and $f_h=\eta e^{\mathsf ix\cdot\xi/h}$, the cross term is
		\[
		\mathcal E_a(f_h,\phi)=-\int_\Omega\overline{\phi(z)}\int_U f_h(y)K_a(z-y)\,dy\,dz.
		\]
		Set
		\[
		I_h(z):=\int_U f_h(y)K_a(z-y)\,dy,\quad z\in\Omega.
		\]
		We claim that $\|I_h\|_{L^2(\Omega)}\le C_Nh^N$ for every $N\ge0$. Since $\eta$ is compactly supported in $U$, we may write
		\[
		I_h(z)=\int_{\R^n}e^{\mathsf iy\cdot\xi/h}b_z(y)\,dy,\quad b_z(y):=\eta(y)K_a(z-y).
		\]
		For fixed $z\in\Omega$, the function $b_z$ is compactly supported in $\supp\eta$. Moreover, $\xi\cdot\nabla_y e^{\mathsf iy\cdot\xi/h}=\frac{\mathsf i|\xi|^2}{h}e^{\mathsf iy\cdot\xi/h}$ so that $e^{\mathsf iy\cdot\xi/h}=\frac{h}{\mathsf i|\xi|^2}\xi\cdot\nabla_y e^{\mathsf iy\cdot\xi/h}$. Integrating by parts $N$ times in the direction $\xi$ gives
		\[
		I_h(z)=\Big(\frac{\mathsf i h}{|\xi|^2}\Big)^N\int_{\R^n}e^{\mathsf iy\cdot\xi/h}(\xi\cdot\nabla_y)^N b_z(y)\,dy.
		\]
		There are no boundary terms because $b_z$ is compactly supported in $U$. By Leibniz' rule, $(\xi\cdot\nabla_y)^N b_z(y)$ is a finite sum of terms of the form $\partial_y^\alpha\eta(y)\,\partial_y^\beta K_a(z-y)$, $|\alpha|+|\beta|\le N$, with constants depending on $N$ and $\xi$. Since $z\in\Omega$ and $y\in\supp\eta$, the vector $z-y$ stays in a compact subset of $\R^n\setminus\{0\}$. Therefore, all derivatives $\partial^\beta K_a(z-y)$ with $|\beta|\le N$ are uniformly bounded for $z\in\Omega$ and $y\in\supp\eta$. It follows that
		\[
		\sup_{z\in\Omega}\int_{\R^n}|(\xi\cdot\nabla_y)^N b_z(y)|\,dy\le C_N.
		\]
		Consequently, $|I_h(z)|\le C_Nh^N$, for $z\in\Omega$. Since $\Omega$ is bounded, this implies
		\[
		\|I_h\|_{L^2(\Omega)}\le C_Nh^N.
		\]
		Thus, for $\phi\in C_c^\infty(\Omega)$,
		\[
		|\mathcal E_a(f_h,\phi)|\le \|I_h\|_{L^2(\Omega)}\|\phi\|_{L^2(\Omega)}
		\le C_Nh^N\|\phi\|_{H^s(\R^n)}.
		\]
		By density, the same estimate holds for all $\phi\in\widetilde H^s(\Omega)$, and this implies 
		\[
		\|\mathcal E_a(f_h,\cdot)\|_{(\widetilde H^s(\Omega))^*}\le C_Nh^N.
		\]
		Using \eqref{eq:wf-equation-overlap} with $\phi=w_{f_h}$ and the coercivity of $\mathcal E_a$ on $\widetilde H^s(\Omega)$, we obtain
		\[
		\|w_{f_h}\|_{H^s(\R^n)}\le C_Nh^N.
		\]
		
		For $x\in K_0$, one has $R_a^Uf_h(x)=L_aw_{f_h}(x)$. Since $w_{f_h}=0$ in $\Omega_e$,
		\[
		L_aw_{f_h}(x)=-\ip{w_{f_h}}{K_a(x-\cdot)}_{\widetilde H^s(\Omega)\times H^{-s}(\Omega)}.
		\]
		The family $z\mapsto K_a(x-z)$ is bounded in $H^{-s}(\Omega)$ uniformly for $x\in K_0$. Hence,
		\[
		|L_aw_{f_h}(x)|\le C\|w_{f_h}\|_{H^s(\R^n)},
		\quad x\in K_0.
		\]
		It follows that
		\[
		\|R_a^Uf_h\|_{L^2(K_0)}
		\le C\|w_{f_h}\|_{H^s(\R^n)}.
		\]
		Together with the estimate for $w_{f_h}$, this proves \eqref{eq:offdiag-smoothing-estimate}.
	\end{proof}
	
	\begin{lemma}[Semiclassical extraction of the symbol]\label{lem:semiclassical-symbol}
		Let $a\in C^\infty(\Sn)$ be a smooth elliptic angular density. Let $\eta\in C_c^\infty(\R^n)$ and $\xi\in\R^n\setminus\{0\}$. Set $f_h(x)=\eta(x)e^{\mathsf i x\cdot\xi/h}$, then
		\begin{equation}\label{eq:semiclassical-symbol-limit}
			\lim_{h\to0}h^{2s}\mathcal E_a(f_h,f_h)=A_a(\xi)\|\eta\|_{L^2(\R^n)}^2.
		\end{equation}
	\end{lemma}
	
	\begin{proof}
		By Lemma \ref{lem:fourier-representation},
		\[
		h^{2s}\mathcal E_a(f_h,f_h)=h^{2s}\int_{\R^n}A_a(\zeta)|\widehat\eta(\zeta-\xi/h)|^2\,d\zeta.
		\]
		We make the change of variables $\rho=\zeta-\xi/h$, so that $\zeta=\xi/h+\rho$, then
		\[
		h^{2s}\mathcal E_a(f_h,f_h)=h^{2s}\int_{\R^n}		A_a(\xi/h+\rho)|\widehat\eta(\rho)|^2\,d\rho.
		\]
		Since $A_a$ is homogeneous of degree $2s$,
		\[
		h^{2s}A_a(\xi/h+\rho)=h^{2s}A_a\big((\xi+h\rho)/h\big)=A_a(\xi+h\rho).
		\]
		Therefore,
		\[
		h^{2s}\mathcal E_a(f_h,f_h)=\int_{\R^n}A_a(\xi+h\rho)|\widehat\eta(\rho)|^2\,d\rho.
		\]
		For $0<h<1$, $|A_a(\xi+h\rho)|\le C|\xi+h\rho|^{2s}\le C_\xi(1+|\rho|)^{2s}$. Since $\widehat\eta$ is rapidly decreasing, the function $C_\xi(1+|\rho|)^{2s}|\widehat\eta(\rho)|^2$ is integrable on $\R^n$. Also, $A_a(\xi+h\rho)\to A_a(\xi)$ pointwise as $h\to0$. Dominated convergence theorem gives
		\[
		\lim_{h\to0}h^{2s}\mathcal E_a(f_h,f_h)=\int_{\R^n}A_a(\xi)|\widehat\eta(\rho)|^2\,d\rho
		=A_a(\xi)\|\eta\|_{L^2(\R^n)}^2,
		\]
		where the last equality follows from Plancherel's theorem.
	\end{proof}
	
	\begin{proof}[Proof of Theorem \ref{thm:overlap}]
		Choose a nonempty open set $U$ such that $U\Subset W_1\cap W_2$. By Lemma \ref{lem:restriction-smaller-windows}, the assumption \eqref{eq:overlap-DN-equality-intro} implies
		\[
		\Lambda_{a_1}^{U,U}f=\Lambda_{a_2}^{U,U}f
		\quad \text{for all }f\in\widetilde H^s(U).
		\]
		Let $\eta\in C_c^\infty(U)$ be nonzero, fix $\xi\in\R^n\setminus\{0\}$, and define $f_h(x)=\eta(x)e^{\mathsf i x\cdot\xi/h}$, for $0<h<1$. Let $K_0\Subset U$ be a compact set such that $\supp\eta\subset K_0$. For $j=1,2$, the off-diagonal smoothing lemma gives
		\[
		\Lambda_{a_j}^{U,U}f_h=L_{a_j}f_h|_U+R_{a_j}^Uf_h\quad \text{in }\mathcal D'(U),
		\]
		where, for every integer $N\ge0$,
		\[
		\|R_{a_j}^Uf_h\|_{L^2(K_0)}\le C_Nh^N.
		\]
		Pairing the previous identity with $f_h$ gives
		\[
		\big\langle\Lambda_{a_j}^{U,U}f_h, f_h\big\rangle=\ip{L_{a_j}f_h}{f_h}+\int_U R_{a_j}^Uf_h\,\overline{f_h}\,dx.
		\]
		Since $\supp f_h\subset K_0$, the remainder term satisfies
		\[
		\bigg|\int_U R_{a_j}^Uf_h\,\overline{f_h}\,dx\bigg|\le \|R_{a_j}^Uf_h\|_{L^2(K_0)}\|\eta\|_{L^2(\R^n)}
		\le C_Nh^N.
		\]
		In particular,
		\[
		h^{2s}\int_U R_{a_j}^Uf_h\,\overline{f_h}\,dx\to0
		\quad \text{as }h\to0.
		\]
		Moreover, by the definition of $L_{a_j}$ through the energy form, $\ip{L_{a_j}f_h}{f_h}=\mathcal E_{a_j}(f_h,f_h)$. Therefore,
		\begin{equation}\label{eq:overlap-proof-limit}
			\lim_{h\to0}h^{2s}\big\langle \Lambda_{a_j}^{U,U}f_h, f_h\big\rangle=\lim_{h\to0}h^{2s}\mathcal E_{a_j}(f_h,f_h).
		\end{equation}
		By Lemma \ref{lem:semiclassical-symbol},
		\[
		\lim_{h\to0}h^{2s}\mathcal E_{a_j}(f_h,f_h)=A_{a_j}(\xi)\|\eta\|_{L^2(\R^n)}^2.
		\]
		Combining this with \eqref{eq:overlap-proof-limit}, we obtain $\lim_{h\to0}h^{2s}\big\langle\Lambda_{a_j}^{U,U}f_h,f_h\big\rangle=A_{a_j}(\xi)\|\eta\|_{L^2(\R^n)}^2$. The restricted DN maps agree on $U$, hence the left hand side is the same for $j=1$ and $j=2$. Since $\eta\not\equiv0$, it follows that
		\[
		A_{a_1}(\xi)=A_{a_2}(\xi)\quad \text{for every }\xi\in\R^n\setminus\{0\}.
		\]
		By homogeneity and continuity of the symbols, the equality also holds at $\xi=0$. Therefore,
		\[
		\int_{\Sn}|\xi\cdot\theta|^{2s}(a_1-a_2)(\theta)\,d\theta=0	\quad \text{for all }\xi\in\R^n.
		\]
		Since $a_1-a_2$ is smooth and even, Lemma \ref{lem:smooth-cosine-injective} gives $a_1=a_2$ on $\Sn$.
	\end{proof}
	
	\begin{proof}[Proof of Theorem \ref{thm:separated-main}]
		Fix $f\in C_c^\infty(W_1)$. For $j=1,2$, let $u_j\in H^s(\R^n)$ be the solution of
		\begin{equation}\label{eq:main-proof-dirichlet}
			\begin{cases}
				L_{a_j}u_j=0 & \text{in }\Omega,\\
				u_j=f & \text{in }\Omega_e.
			\end{cases}
		\end{equation}
		Since $\supp f\subset W_1$ and $\overline W_1\cap\overline W_2=\emptyset$, the exterior condition gives $u_1=u_2=0$ in $W_2$. By Lemma \ref{lem:DN-bounded} and the hypothesis \eqref{eq:operator-DN-equality-main},
		\begin{equation}\label{eq:La1La2-in-W2}
			L_{a_1}u_1=L_{a_2}u_2
			\quad\text{in }\mathcal D'(W_2).
		\end{equation}
		Applying the local operator $(-\Delta)^m$ to \eqref{eq:La1La2-in-W2}, we get $(-\Delta)^mL_{a_1}u_1=(-\Delta)^mL_{a_2}u_2$ in $\mathcal D'(W_2)$. By Lemma \ref{lem:factorization-distributions}, this gives
		\[
		(-\Delta)^s(Q_{a_1}(D)u_1-Q_{a_2}(D)u_2)=0\quad \text{in } \mathcal D'(W_2).
		\]
		Set
		\[
		w:=Q_{a_1}(D)u_1-Q_{a_2}(D)u_2.
		\]
		Since $u_j\in H^s(\R^n)$ and $Q_{a_j}(D)$ has order $2m$, one has $w\in H^{s-2m}(\R^n)$. The previous identity says that $(-\Delta)^s w=0$ in $\mathcal D'(W_2)$. On the other hand, $w=0$ in $W_2$, because $u_1=u_2=0$ in $W_2$ and $Q_{a_j}(D)$ are local differential operators. Theorem \ref{thm:fractional-UCP} gives $w=0$ in $\R^n$. Restricting this identity to $W_1$ and using $u_1=u_2=f$ in $W_1$, we obtain $Q_{a_1}(D)f=Q_{a_2}(D)f$ in $W_1$. Since $f\in C_c^\infty(W_1)$ was arbitrary, Lemma \ref{lem:local-operator-id} gives $Q_{a_1}=Q_{a_2}$. By the injectivity of $a\mapsto Q_a$ in Lemma \ref{lem:factorization}, $a_1=a_2$.
	\end{proof}

	\begin{proof}[Proof of Theorem \ref{thm:analytic-main}]
		Let $S_j$ denote the off-diagonal DN kernel associated with $a_j$, and set $S_{12}:=S_1-S_2$. By Lemma \ref{lem:off-diagonal-kernel}, the operator $\Lambda_{a_1}^{W_1,W_2}-\Lambda_{a_2}^{W_1,W_2}$ has distribution kernel $S_{12}$ on $W_2\times W_1$. For $f\in C_c^\infty(W_1)$ and $g\in C_c^\infty(W_2)$, the equality \eqref{eq:analytic-DN-equality-main} gives
		\[
		0=\ip{(\Lambda_{a_1}^{W_1,W_2}-\Lambda_{a_2}^{W_1,W_2})f}{g}
		=\int_{W_2}\int_{W_1}S_{12}(x,y)f(y)\overline{g(x)}\,dy\,dx.
		\]
		Finite sums of functions of the form $(x,y)\mapsto \overline{g(x)}f(y)$ are dense in $C_c^\infty(W_2\times W_1)$. Hence, $S_{12}=0$ as a distribution on $W_2\times W_1$. Since $S_{12}$ is real analytic there by Lemma \ref{lem:analytic-kernel}, it follows that $S_{12}$ vanishes pointwise on $W_2\times W_1$.
		
		By Lemma \ref{lem:near-to-far}, $S_{12}=0$ on the full exterior pair domain $\mathcal D_G$. Since $\Omega$ is bounded, there exists $R_0>0$ such that $\R^n\setminus B_{R_0}\subset G$. Fix $\theta\in\Sn$. Choose $e\in\Sn$ with $e\cdot\theta=0$, and set $p=e+\theta$ and $q=e$. Then $p-q=\theta$. For every sufficiently large $R$, the points $Rp$ and $Rq$ belong to $G$, and $Rp\ne Rq$. Hence
		\begin{equation}\label{eq:analytic-proof-zero}
			S_{12}(Rp,Rq)=0.
		\end{equation}
		Using the kernel representation \eqref{eq:Sa-kernel}, we have
		\[
		S_{12}(x,y)=-K_{a_1-a_2}(x-y)+\big(S_{a_1}(x,y)+K_{a_1}(x-y)\big)-\big(S_{a_2}(x,y)+K_{a_2}(x-y)\big).
		\]
		Since $p\ne0$ and $q\ne0$, and since $\Omega$ is bounded, there are constants $c,C>0$ such that
		\[
		cR\le \dist(Rp,\Omega),\dist(Rq,\Omega)\le CR
		\]
		for all sufficiently large $R$. By Lemma \ref{lem:far-field-asymptotic}, the last two terms are $O(R^{-2n-4s})$ when $(x,y)=(Rp,Rq)$. Hence, \eqref{eq:analytic-proof-zero} gives
		\begin{equation}\label{eq:K+R_asymp}
			0=-K_{a_1-a_2}(R\theta)+O(R^{-2n-4s}).
		\end{equation}
		Since $K_{a_1-a_2}(R\theta)=R^{-n-2s}(a_1-a_2)(\theta)$, multiplying by $R^{n+2s}$ and sending $R\to\infty$,  \eqref{eq:K+R_asymp} gives $(a_1-a_2)(\theta)=0$. Since $\theta\in\Sn$ was arbitrary, $a_1=a_2$ on $\Sn$.
	\end{proof}
	
	\begin{corollary}[Arbitrary exterior open sets in the finite harmonic class]\label{cor:any-exterior-sets}
		Let $0<s<1$ and $n\ge2$. Let $\Omega\subset\R^n$ be bounded and open. Let $W_1,W_2\Subset\Omega_e$ be nonempty open sets. Let $a_1,a_2\in\mathcal A_m$ be admissible finite harmonic densities of the same order $m$. If $\Lambda_{a_1}^{W_1,W_2}f=\Lambda_{a_2}^{W_1,W_2}f$ for all $f\in \wt H^s(W_1)$, then $a_1=a_2$.
	\end{corollary}
	
	\begin{proof}
		Choose $x_1\in W_1$ and $x_2\in W_2$ with $x_1\ne x_2$. This is possible because nonempty open subsets of $\R^n$ contain more than one point. Since $W_1$ and $W_2$ are open, we may choose nonempty open balls $U_1\Subset W_1$ and $U_2\Subset W_2$ around $x_1$ and $x_2$, respectively, so small that $\overline U_1\cap\overline U_2=\emptyset$. By Lemma \ref{lem:restriction-smaller-windows}, the equality of the DN maps on $W_1,W_2$ restricts to the equality on $U_1,U_2$. Theorem \ref{thm:separated-main} gives $a_1=a_2$.
	\end{proof}

	\section*{Statements and Declarations}
	
	\noindent\textbf{Data availability statement.}
	No datasets were generated or analyzed during the current study.
	
	\medskip
	
	\noindent\textbf{Conflict of Interests.} The author declares no competing interests.

	\medskip

	\noindent\textbf{Acknowledgments.}
	The author is partially supported by the National Science and Technology Council (NSTC) of Taiwan, under the project 113-2628-M-A49-003. The author also acknowledges financial support from the Alexander von Humboldt Foundation through the Henriette Herz Scouting Programme, hosted by Universität Duisburg-Essen.

	\bibliographystyle{alpha}
	\bibliography{refs}
	
\end{document}